\documentclass[11pt]{article}

\usepackage[english]{babel}
\usepackage[utf8]{inputenc}
\usepackage[T1]{fontenc}
\usepackage{lmodern}
\usepackage{graphicx}
\usepackage{amsfonts}
\usepackage[nice]{nicefrac}
\usepackage{amsmath,amsthm,amssymb} 
\usepackage{thmtools}
\usepackage{caption}
\usepackage{array}
\usepackage{color}
\usepackage{url}
\usepackage[top=3cm,bottom=3.5cm,right=2cm,left=2cm]{geometry}
\usepackage{enumerate}
\usepackage{subfig}

\usepackage[affil-it]{authblk}

\newcommand{\sss}{\scriptscriptstyle}

\def \Trees{{\sf Tree}}
\def \IBT{{\sf IBT}}
\def \CBT{{\sf CBT}}
\def \Rec{{\sf Rec}}
\def \BST{{\sf BST}}
\def \CBST{{\sf CBST}}

\def \RRT{{\sf RRT}}
\def \ba{\begin{align}}
\def \ea{\end{align}}
\def \be{\begin{eqnarray*}}
\def \ee{\end{eqnarray*}}
\def \ben{\begin{eqnarray}}
\def \een{\end{eqnarray}}

\def \floor#1{\lfloor #1\rfloor}
\def \beq{\begin{equation}}
\def \eq{\end{equation}}
\def \eref#1{(\ref{#1})}

\def \l{\left}
\def \r{\right}
\def \as{\xrightarrow[n]{(as.)}}
\def \dd{\xrightarrow[n]{(d)}}
\def \wt#1{\widetilde{#1}}

\newcommand{\N}{\mathbb{N}}
\newcommand{\C}{\mathbb{C}}
\newcommand{\R}{\mathbb{R}}

\def\ovs#1{#1^\star}
\def \ov#1{\overline{#1}}
\font\dsrom=dsrom10 scaled 1200
\def \indi{\textrm{\dsrom{1}}}

\newcommand{\Var}{\mathtt{Var}}

\newcommand{\E}{\mathbb{E}}
\renewcommand{\P}{\mathbb{P}}
\newcommand{\PP}{\mathcal P}
\newcommand{\RR}{\mathcal R}
\newcommand{\TR}{{\sf TR}}

\newcounter{a}

\newenvironment{Hyp}
{\refstepcounter{a} 
\hfill\begin{minipage}{\dimexpr\textwidth-1.5em} {\bf Hyp \thea:}}
{\end{minipage}}
\def \dis{\displaystyle}
\def \H {{Hyp }}

\def \Pol{P{\'o}lya }
\def \MM{{\mathcal M}}
\newtheorem{thm}{Theorem}
\newtheorem{prop}[thm]{Proposition}
\newtheorem{df}[thm]{Definition}
\newtheorem{lem}[thm]{Lemma}

\newtheorem{cor}[thm]{Corollary}
\newtheorem{rem}[thm]{Remark}

\def \sur#1#2{\mathrel{\mathop{\kern 0pt#1}\limits^{#2}}}

\def \eqd{\sur{=}{(d)}}
\def \dd{\xrightarrow[n]{(d)}}

\def \proba{\xrightarrow[n]{(proba.)}}

\makeatletter 
\renewenvironment{proof}[1][Proof] {\par\pushQED{\qed}\normalfont\topsep6\p@\@plus6\p@\relax\trivlist\item[\hskip\labelsep\bfseries#1\@addpunct{.}]\ignorespaces}{\popQED\endtrivlist\@endpefalse} 
\makeatother

\title{Measure-valued \Pol urn processes}
\author{C{\'e}cile Mailler\thanks{Department of Mathematical Sciences, University of Bath, Claverton Down, BA2 7AY Bath, UK. c.mailler@bath.ac.uk}~~and Jean-François Marckert\thanks{CNRS, LaBRI ,Universit\'e Bordeaux, 351 cours de la Libération  33405 Talence cedex, France}}

\begin{document}
\maketitle

\begin{abstract} 
A \Pol urn process is a Markov chain that models the evolution of 
an urn containing some coloured balls, 
the set of possible colours being $\{1,\ldots,d\}$ for $d\in \mathbb{N}$.
At each time step, a random ball is chosen uniformly in the urn.
It is replaced in the urn and,
if its colour is $c$, $R_{c,j}$ balls of colour $j$ are also added (for all $1\leq j\leq d$).

We introduce a model of measure-valued processes that generalises this construction.
This generalisation includes the case when the space of colours is a (possibly infinite) Polish space $\PP$. 
We see the urn composition at any time step $n$ as a measure ${\cal M}_n$ -- possibly non atomic -- on $\PP$. 
In this generalisation, we choose a random colour $c$ according  to the probability distribution proportional to ${\cal M}_n$, 
and add a measure ${\cal R}_c$ in the urn, where 
the quantity ${\cal R}_c(B)$ of a Borel set~$B$
models the added weight of ``balls'' with colour in~$B$.

We study the asymptotic behaviour of these measure-valued \Pol urn processes, 
and give some conditions on the replacements measures $({\cal R}_c,c\in \PP)$ 
for the sequence of measures $({\cal M}_n, n\geq 0)$ to converge in distribution, possibly after rescaling. 
For certain models, related to branching random walks, 
$({\cal M}_n, n\geq 0)$ is shown to converge almost surely under some moment hypothesis;
a particular case of this last result gives the almost sure convergence of the (renormalised) 
profile of the random recursive tree to a standard Gaussian.
\end{abstract}


\paragraph{Acknowledgement :}The first author is grateful to EPSRC for support through the grant EP/K016075/1. The second author has been partially supported by ANR-14-CE25-0014 (ANR GRAAL).

\section{Introduction}

\subsection{The $d$--colour \Pol urn process}

A \Pol urn process is a simple time-homogeneous Markov chain $(M_n, n\geq 0)$ on $\mathbb{N}^d$ 
that models the evolution of an urn containing some coloured balls, 
the set of possible colours being $\{1, \ldots, d\}$ for $d\in \mathbb N$. 
For all integers~$n$ and for all $j\in\{1, \ldots, d\}$, $M_{nj}\geq 0$ is the number of balls of colour $j$ in the urn at time~$n$,
and $M_n=(M_{n1},\dots,M_{nd})$ is the urn composition at time $n$.

A \Pol urn is defined by two parameters: 
an initial composition $M_0$ and a replacement matrix $R=(R_{i,j})_{1\leq i,j\leq d}$ where the $R_{i,j}$ are integers. 
The initial urn composition $M_0$ is a vector with non-negative entries such that the initial total number of balls in the urn is 
both positive and finite -- in other words, $0<\|M_0\|_1=\sum_j M_{0j}<+\infty$  almost surely (a.s.).

The Markov chain $(M_n)_{n\geq 0}$ evolves as follows: 
At time $n$, pick a ball uniformly at random among the balls in the urn; 
conditionally on $M_n$, the distribution of the colour $C_n$ of the picked ball verifies
\ben \label{eq:choose}
\mathbb P(C_{n}=i~|~M_n)= {M_{ni}}\,/\,{\|M_n\|_1}.
\een
Conditionally on $C_n=C$, the composition vector evolves as follows:
\ben\label{eq:replace}
M_{n+1} = M_n + R_C
\een
where $R_i= (R_{i,1}, \ldots, R_{i,d})$ is the $i^{\text{th}}$ line of the replacement matrix. 
In other words, the picked ball is replaced into the urn, and $R_{C,j}$ new balls of colour~$j$ are added, for every $j\in\{1, \ldots, d\}$. 
Authors are often interested in the asymptotic behaviour of the urn composition when time goes to infinity 
and many results have been obtained for various cases 
(see e.g. Janson \cite{Janson2004177}, Flajolet \& al. \cite{flajolet2005} and references therein).
In general it is assumed that the urn is \emph{tenable}, i.e.\ that
\ben
R_{i,j}\geq -\indi_{i=j} \quad(\forall~ 1\leq i\leq d),
\een
which allows to remove the picked ball from the urn 
but ensures that no impossible configuration occurs, 
i.e.\ that the number of balls of each colour stays non-negative.

Following the standard terminology, 
the $R$ is said irreducible if for any $i, j \in\{1, \ldots, d\}$, 
there exists $n>1$ such that~$R^n_{i,j}>0$. 
   
An important result on the asymptotic behaviour of $d$-colour urns is the following one:
\begin{thm}[{see e.g. Janson~\cite[Theorem 3.1]{Janson2004177} or Athreya and Ney~\cite{AN}}]\label{thm:Jan}
If $R$ is irreducible and tenable, then the largest eigenvalue $\lambda_1$ of $R$ is positive.
If we denote by $v_1$ the left eigenvector of~$R$ associated to~$\lambda_1$ such that $\|v_1\|_1 = 1$, then
for any $0<\|M_0\|_1<+\infty$,
\[\frac{M_n}{n} \to \lambda_1 v_1, ~~\textrm{ almost surely}.\]
\end{thm}

\subsection{The main ideas and results in this paper}
In this paper, we introduce a new point of view on \Pol urn processes: 
we propose viewing the urn composition as a  finite positive measure  $\mu$ on a general colour set (a Polish space $\PP$):
For all Borel sets~$B$, $\mu(B)$ stands for the mass of balls that have colour in~$B$.  We do not restrict ourselves to atomic measures (sum of Dirac measures which corresponds to standard \Pol urn processes), and thus it is possible that no singleton has positive mass.  

Picking a colour randomly is replaced by picking a random colour $C$ according to the probability distribution proportional to $\mu$ (that is $\mu /\mu(\PP)$). 
When the colour $C$ is drawn, then the new urn composition becomes $\mu+\RR_C$ where $\RR_C$ is a finite positive measure on $\PP$ which depends on $C$.

This approach -- which was not needed to treat $d$-colour \Pol urn processes -- is, in our opinion, the right generalisation of \Pol urn processes. It provides a suitable technical framework that, on the one hand, allows infinitely many colours (countable or not), and, on the other hand, allows one to define ``non-atomic'' \Pol urn process. 

The importance of extending P\'olya urn processes to infinite settings was highlighted by Janson,
although up till now it was ``\emph{far from clear how such an extension should be formulated}'' (see~\cite[Remark~4.1]{Janson2004177}). 
Janson also gives three examples of infinitely-many-colour \Pol urns, 
the first two are solvable by chance (Examples~7.5 and~7.6), and the last one (Example~7.9),
which involves a branching random walk on an infinite group, 
is stated as an open problem that falls in our setting.

The present paper shows how to extend P\'olya urn processes to infinite settings
 by considering \emph{measure-valued P{\'o}lya processes};
we prove some asymptotic results in this general framework. 
The construction we provide goes far beyond a simple generalisation of \Pol urn processes to infinitely-many colours since we allow the colour set to be uncountable and the balls to be infinitesimal.
Indeed, we take the point of view of probability theory, and describe the urn composition by a general measure (possibly non-atomic) on the set of colours.

Our work was partially motivated by Bandyopadhyay \& Thacker~\cite{BT1}. This paper treats a very particular case where the set of colours is the integer line~$\mathbb Z$; in~\cite{BT3}, the authors give more detailed results about this model (rate of convergence and large deviations).
In their very recent article~\cite{BT2}, they generalise this example to a wider class.  
Similarly to what we do in this article, 
they encode the P\'olya urn by a branching Markov chain built on a random recursive tree
(this is already present in a restrictive form in their first article).
{However, the results they prove need more restrictive assumptions than the ones proved here.}
We compare in detail Bandyopadhyay \& Thacker's results with ours at the end of Section~\ref{sub:mainth}.

\vspace{\baselineskip}

In the rest of this introduction, we define our measure-valued P\'olya processes (MVPPs) and state our main results (namely Theorems~\ref{th:main},~\ref{theo:BRW} and~\ref{th:main-2} below).
In Section~\ref{sub:mainth}, we encode each MVPP by a branching Markov chain and state
Theorem~\ref{th:main}, which gives the convergence in probability of the composition measure of a MVPP 
under some assumptions on the replacement measures $(\mathcal R_x, x\in\PP)$.
In Section~\ref{sub:as}, we state Theorem~\ref{theo:BRW}, 
which gives almost sure convergence of the composition measure for a certain class of measure-valued MVPPs 
(namely the MVPPs associated to a simple branching random walk with strong moment conditions on the increments).
Finally, in Section~\ref{sub:without}, 
we define a slightly different model that allows us to consider drawing without replacement
and state convergence in probability for this alternative model in Theorem~\ref{th:main-2}.

\subsection{Definition of our measure-valued \Pol urn process}\label{sub:mainth}
Throughout the paper, $\PP$ denotes the colour set; it is a general Polish space.

We introduce the measure-valued \Pol process $(\MM_n)_{n\geq 0}$ (MVPP) as follows:
for all $n\geq 0$, $\MM_n$ is a non-negative Borel measure on $\PP$.
For all Borel sets $B\in\cal B(\PP)$,
$\MM_n(B)$ represents the mass of balls whose colours belong to~$B$. 
The urn process $(\MM_n)_{n\geq 0}$ depends on two parameters:
an initial composition $\mathcal M_0$ which is a non-negative distribution on $\PP$,
and a family $(\RR_x, x \in \PP)$ of non-negative Borel measures, called the \it replacement measures\rm.\par
The mass $\MM_n(\PP)$ can be interpreted as the total mass of balls at time $n$. In the countable case, it would be the total number of balls in the urn, but in our framework, $\MM_n(\PP)$ is not assumed to be an integer. 
Picking a ball uniformly at random at time $n$ in the countable case 
is replaced by the following procedure: 
Pick a random colour $C_n$ under the probability distribution ${\sf Nor}(\MM_n)$, 
where for all finite measure~$\mu$ on~$\mathcal P$, ${\sf Nor}(\mu)$ is the probability distribution proportional to $\mu$: 
\ben\label{eq:norm}{\sf Nor}(\mu):= \frac{\mu}{\mu(\mathcal P)}.\een
Conditionally on $C_n=C$, the composition of the urn at time $n+1$ is given by
\ben\label{eq:add}
\MM_{n+1}=\MM_{n}+ \RR_C.
\een
Recall that $\RR_C$ is a Borel measure: for any Borel set~$B$, 
$\RR_C(B)$ encodes the mass of balls of colour in~$B$ added in the urn 
when a ball of colour~$C$ has been drawn.

The process $({\cal M}_n,n\geq 0)$ is still a time-homogeneous Markov chain. 
Given an initial measure ${\cal M}_0$ and a replacement kernel $(\RR_x,x\in \PP)$, 
we will say that $({\cal M}_n,n\geq 0)$ is a $({\cal M}_0,(\RR_x,x\in \PP))$-MVPP.

One can check that a $d$-colour \Pol urn process is a MVPP by letting 
$\MM_n = \sum_{x\in \Upsilon} M_{nx}\, \delta_{x}$ where $\delta_x$ is the Dirac measure at $x$, 
and $\RR_x = \sum_{y\in \Upsilon} R_{x,y} \delta_y$, where $\Upsilon=\{1, \ldots, d\}$.
Note that taking $\Upsilon$ being a countable set instead of $\{1, \ldots, d\}$ 
gives a \Pol urn process with infinitely (but countably) many colours.

Throughout the paper we assume that:~\\

\begin{Hyp} \label{hyp:balance}
For all $x\in \PP$, $\RR_x$ is a non negative measure on $\mathcal P$ with total mass $\RR_x(\PP)=1$. 
\end{Hyp}~\\

Actually, we only need to assume that $\RR_x(\PP)$ does not depend on~$x$, 
but assuming that it is equal to~1 makes no loss of generality.
Indeed, if we consider the two families of replacement kernels $(\RR_x, x\in \PP)$ 
and $(\RR'_x = c \RR_x, x\in \PP)$, and the two MVPP $(\MM_n,n\geq 0)$ and $(\MM_n',n\geq 0)$ they define, 
we have
\[(\MM'_n,n\geq 0)\eqd c\,(\MM_n,n\geq 0),\quad\text{ if }\quad \MM'_0\eqd c \MM_0.\]

Note that \H\ref{hyp:balance} is equivalent to the \emph{balance} condition in the study of standard \Pol urn processes. 
Indeed, in the $d$-colour case, an urn is balanced if there exists 
an integer~$S$ such that, for all $1\leq i\leq d$, $\sum_{j=1}^d R_{i,j} = S$, implying that the total number of balls in the urn at time~$n$ is~$nS$ plus the number of balls already in the urn at time 0.

We want to design some sufficient conditions on the family $\RR$ to ensure the convergence of $\MM_n$ after normalisation (for some initial measure $\MM_0$). 
Before stating our results, let us give the intuitive ideas underlying our approach.
Consider a MVPP as defined above, and consider the successive drawn colours $(C_i,i\geq 1)$. At time $n$, the identity
\ben\label{eq:dec}
\MM_n=\MM_0+\sum_{i=1}^n \RR_{C_i}\een
shows that the sequence of drawn colours determines the sequence $(\MM_n,n\geq 0)$. 
Further, to choose a random colour $C$ according to ${\sf Nor}({\cal M}_n)$ can be represented as follows:
\begin{itemize}
\item[$(a)$] with probability ${\cal M}_0(\PP)/{\cal M}_n(\mathcal P)$ sample $C_{n+1}$ according ${\sf Nor}({\cal M}_0)$,
\item[$(b)$] with probability $1/{\cal M}_n(\PP)$ sample $C_{n+1}$ according to $\RR_{C_i}$ (for any $1\leq i \leq n$);
\end{itemize}
or replace $(b)$ by $(b')$:
\begin{itemize}
\item[$(b')$] choose $U_{n+1}$ uniform in $\{1, \ldots, n\}$ then sample $C_{n+1}$ according to $\RR_{C_{U_{n+1}}}$. 
\end{itemize}

Replacing $(b)$ by $(b')$ makes the branching structure of the MVPP visible:
$\mathcal M_n$ is a sum of $n+1$ distributions, and one can consider that the term $\mathcal R_{C_{n+1}}$ added at time $n+1$ is the ``child'' of the term $\RR_{C_{U_{n+1}}}$, which was drawn uniformly (up to the biased weight of the $\mathcal M_0$-term) at random among the terms of $\mathcal M_n$.
Recursively, the evolution of ${\cal M}_n$ (up to considerations involving ${\cal M}_0$) appears to be perfectly encoded by a random recursive tree, 
and this fact is at the heart of our analysis.

We now introduce a Markov chain defined on $\PP$, which will be used to express our convergence result.

\vspace{\baselineskip}
{\bf The companion Markov chain --}~\\ 
Given the pair $\l(\MM_0, \RR\r)$ (that defines the MVPP $(\mathcal M_n)_{n\geq 0}$) 
we define the Markov chain $(W_n)_{n\geq 0}$ on $\PP$ as follows:
\begin{itemize}
\item The initial distribution of $W_0$ is $\mu_0= {\sf Nor}(\MM_0)$.
\item The Markov kernel of this Markov chain is defined for any $(x,A)\in \PP \times {\cal B}(\PP)$ by
\ben\label{eq:Markov-kernel}
K(x,A)= \RR_x(A).
\een
\end{itemize}
In other words: assume that $(W_m)_{m\leq n}$ has been defined. Conditionally on $W_n=w$, 
$W_{n+1}$ is defined as a random variable with law $\RR_{w}$.

The two processes $(W_n)_{n\geq 0}$ and $(\MM_n)_{n\geq 0}$ are very different since the first one is a $\PP$-valued Markov chain, with Markov kernel $K$, and the second one is a Markov chain with values in ${\cal M}^+(\PP)$ the set of non-negative Borel measures on $\PP$. 

\begin{df} 
We say that a Markov chain $(X_n)_{n\geq 0}$ with initial distribution $\mu_0$ is {\bf $\big(a(n), b(n)\big)_{n\geq 0}$ convergent} if the sequence $\l(\frac{X_n-b(n)}{a(n)}\r)_{n\geq 0}$ converges in distribution to some distribution $\mu_{\infty}$ (which may depends on $\mu_0$). It is said to be {\bf $\big(a(n), b(n)\big)_{n\geq 0}$ ergodic} if it is $\big(a(n), b(n)\big)_{n\geq 0}$ convergent for any initial distribution $\mu_0$, and if the limiting distribution $\mu_{\infty}$ does not depend on $\mu_0$.
\end{df}
Note that the $(1,0)_{n\geq 0}$ convergence is the simple convergence in distribution. 
\begin{rem} 
When working on a general Polish space $\PP$, 
subtracting $b(n)\in\PP$ to $X_n$ and dividing by $a(n)$ might have no meaning.
If $\PP$ is not equipped with a subtraction operation 
(which may be different from the usual notion of difference -- this is just a binary operation on $\PP$), 
the only meaningful choice for $b(n)$ is $0$ and we set the convention $X_n - 0:= X_n$.

When $a(n)=1$, we set $x/1:= x$ for all $x\in \PP$
(even if the division by $1$ is not well defined on the space). 
If $a(n)$ is not 1, the elements of $\big(a(n), n\geq 0\big)$ belong to a set~$K$ such that 
the ``division'' of the elements of $\PP$ by those of $K$ is well defined 
(for example,  if $\PP$ is the set of $3\times 3$ matrices with complex coefficients, 
$K$ can be $\mathbb{R}\setminus\{0\}$). 
We also need the quotient of two elements of $K$ to be well defined. 

In most of our examples, $\PP$ will be a Banach space (on $\R$ or $\C$), on which subtraction and division by a scalar are well defined.
\end{rem}

For any measure $\mu$, for any scalar $a$ and any $b\in\PP$, 
denote by $\Theta_{a,b}(\mu)$ the measure defined by 
\ben
\int_\PP f\, d\Theta_{a,b}(\mu) := \int_\PP f\l(a^{-1}\l(x-b\r)\r)\,d\mu(x),
\een
for all measurable functions~$f$.
If $\mu$ is the probability distribution of a random variable $X$, 
then $\Theta_{a,b}(\mu)$ is the distribution of $a^{-1}(X-b)$.

One of the main results of the paper is the following:
\begin{thm} \label{th:main} Assume that \H \ref{hyp:balance} holds,
and that there exists a pair $\big(a(n), b(n)\big)_{n\geq 0}$ satisfying the following constraints:
\begin{itemize}
\item[(a)] the Markov chain $(W_n)_{n\geq 0}$ is $(a(n), b(n))_{n\geq 0}$-ergodic with limiting distribution $\gamma$,
\item[(b)] for any $x \in \mathbb{R}$, for any sequence $\varepsilon_n  =o(\sqrt{n})$,
\ben
\frac{b(\floor{n+x \sqrt{n} + \varepsilon_n}) - b(n)}{a(n)} &\to& f(x) \label{eq:b(n)}\\
\frac{a(\floor{n+x\sqrt{n}+\varepsilon_n})}{a(n)}& \to& g(x) \label{eq:a(n)}
\een
where $f:\R\to\PP$ and $g$ are two measurable functions (pointwise convergence almost everywhere suffices).
\end{itemize}
Under these hypotheses, for any finite measure $\mathcal M_0$ such that $0<\mathcal M_0(\PP)<+\infty$,
we have
\ben \label{eq:gco}
\Theta_{a({\log n}),\, b({\log n})}\l(n^{-1}{\MM_n}\r) \proba  {\nu}
\een
for the topology of weak convergence on ${\mathcal M}(\PP)$, 
{and where $\nu$ is}
the distribution of $\Gamma g(\Lambda)+f(\Lambda)$ where $\Lambda$ is a $\mathcal N(0,1)$ random variable, 
and $\Gamma\sim \gamma$ is independent of~$\Lambda$.
\end{thm} 

\begin{rem} 
In fact, in this theorem and in the rest of the article, as explained in Section~\ref{sec:generalM0}, 
the role played by the initial measure ${\cal M}_0$ is secondary.
\end{rem}

{Bandyopadhyay \& Thacker~\cite{BT2} in their Theorem~3.2, state a similar result but under more restrictive assumptions: in the Polish case for $a(n)=1$ and $b(n)=0$ and in $R^d$ for two special cases of renormalisation sequences $a(n)$ and $b(n)$.}
Bandyopadhyay \& Thacker also give numerous examples (see~\cite[Section~4]{BT2}) to which our result also applies directly.

\subsection{Almost sure convergent MVPPs}\label{sub:as}
As already stated in Theorem~\ref{thm:Jan}, 
almost sure convergence of the rescaled urn composition is already known for $d$-colour P\'olya urns; 
see Athreya and Ney~\cite{AN} or Janson~\cite{Janson2004177}.

In this section, we state almost sure convergence in another case: ``the random walk case'',
which corresponds to the case where
the companion Markov chain is a random walk whose increments have exponential moments.

This random walk case is the case where $\RR_x$ is the law of $x+\Delta$ where $\Delta$ is a random variable (which does not depend on~$x$).
In this case, the underlying Markov chain $(W_n)_{n\geq 0}$ is the simple random walk of increment~$\Delta$.
We are able to prove strong convergence of the (scaled) MVPP when the increments~$\Delta\in\mathbb R^d$ 
have exponential moments in the neighbourhood of~$0$.
Assume that there exists~$r_1>0$, such that
\ben\label{eq:ine} \E[\exp(\theta \Delta)]<+\infty \textrm{ for any } \theta \in \mathcal B(0,r_1),\een
where $\mathcal B(0,r_1)$ is the closed ball centred at the origin and of radius $r_1$.
Note that, by continuity of the Laplace transform, if we denote 
\[S_x=  \sup_{\theta \in \mathcal B(0,x)} \l|\E[\exp(\theta\Delta)]-1\r|,\]
then we have
\ben
\label{eq:ine2} S_x \xrightarrow[x\to 0]{} 0, \textrm{ and }~~ S_{r_1} <+\infty.
\een
\begin{thm}\label{theo:BRW}
Assume that for any $x\in \R^d$, 
$\RR_x$ is the law of $x+\Delta$ where $\Delta$ is a random variable in $\mathbb R^d$ 
(which does not depend on~$x$). 
Assume that $\Delta$ has exponential moments in a neighbourhood of 0, 
and denote by~$m$ its mean and by~$\Sigma^2$ its covariance matrix.
Then, for any finite measure~$\mathcal M_0$ such that $0<\mathcal M_0(\PP)<+\infty$,
\ben\label{eq:cv_BRW}
\Theta_{\sqrt{\log n},\, m\log n}\l(n^{-1}{\mathcal M_n}\r) \as {\mathcal N(0,\,\Sigma^2+m^T m)}
\een
where $m^T$ stands for the transpose of $m$.
\end{thm}

The convergence in probability in this case is a direct consequence of~Theorem~\ref{th:main} 
and has also already been proved by 
Bandyopadhyay \& Thacker~\cite[Theorem 2]{BT1}, 
together with some speed-of-convergence results.
However, the almost sure convergence in Theorem~\ref{theo:BRW} is a new result.
 
The proof of almost sure convergence in this setting
is obtained by proving (by a martingale method) 
that the occupation measure of a branching random walk built on a random recursive tree converges, after normalisation, almost surely. 
A similar result was obtain by Biggins~\cite{biggins1992} for branching random walks on Galton--Watson trees;
both Biggins' result and ours need, for the same reason, the same somewhat-restrictive moment assumption.
The proof we give is very much inspired by that of Chauvin \& al.~\cite{CDJH} (following Joffe, Le Cam \& Neveu~\cite{JLCN}'s method) 
where where they prove the convergence of the profile of binary search trees.

\vspace{\baselineskip}
As a corollary of Theorem~\ref{theo:BRW} 
we obtain a strong convergence result for the profile of the random recursive tree.
The random recursive tree, or rather the sequence of random recursive trees, 
will be defined more formally later in this paper (see Section~\ref{sec:RRDT}). 
It is built as follows: $\RRT_0$ has a unique node being its root;
to build $\RRT_{n+1}$ from $\RRT_n$, 
we pick a node uniformly at random in $\RRT_n$ and add a children to this node.
For any node~$u$, we denote by~$|u|$ the graph distance between~$u$ and the root.
The profile of the random recursive tree~$\RRT_n$ is the measure
\[{\sf Prof}_n := \frac1n \sum_{k} X_{n,k}\delta_{k},\]
where  $X_{n,k}$ is the number of nodes at distance $k$ of the root in $\RRT_n$.
The profile of a tree gives valuable information about its shape and has been studied for various random trees: see for example Drmota \& Gittenberger~\cite{DG97} for the Catalan tree; Chauvin \& al.~\cite{CDJH} for the binary search tree; Schopp~\cite{Schopp} for $m$-ary search trees; Katona~\cite{Katona} and Sulzbach~\cite{Henning} for preferential attachment trees;  Drmota, Janson \& Neininger~\cite{DJN08} for random search trees; and Drmota \& Hwang~\cite{DH05}, Fuchs, Hwang \& Neininger~\cite{Fuchs2006} for the random recursive tree.
In the latter papers the authors prove that if $\nicefrac{k}{\log n}$ converges to $\alpha\geq 0$ then $\frac{X_{n,k}}{\mathbb E X_{n,k}}$ converges in distribution to some limit law $X(\alpha)$. They prove that convergence holds for all moments only if $\alpha\in[0,1]$ and also that if $\alpha =1$ and $|k-\log n|\to\infty$ then $(X_{n,k}-\mathbb EX_{n,k})/(\Var X_{n,k})^{\nicefrac12}$ converges in distribution to a random variable.
As a corollary of Theorem~\ref{theo:BRW}, we are able to give an additional result about the profile of the random recursive tree:
Taking $\Delta = 1$ (the random walk with increment equal to 1 a.s.) and $\mathcal M_0 = \delta_0$ in Theorem~\ref{theo:BRW}, 
we get that  
\ben
{\sf Prof}_n=n^{-1}\mathcal M_n.
\een
As a consequence, we have
\begin{cor} \label{cor:ehhte} The sequence of rescaled profiles converge a.s.:
\ben
\Theta_{\sqrt{\log n},\,\log n}\l({\sf Prof}_n\r)\as {\mathcal N(0,1)}
\een
\end{cor}
Equivalently, let $\overline{{\sf Prof}_n}(x)$ be the proportion of nodes in $\RRT_n$ at distance at most $\log n + x\sqrt{\log n}$ of the root. We have $\overline{{\sf Prof}_n}\as \Phi$ in $D(\mathbb R)$ (the space of c\`ad-l\`ag functions equipped with the Skorokhod's topology) where $\Phi$ is the distribution function of the standard Gaussian distribution.

Note that, although this result is new for the random recursive tree,
a stronger, local result is known for the binary search tree (see~\cite[Theorem~1]{CDJH}) and for preferential attachment trees (see~\cite{Katona}).

\subsection{Drawing without replacement}\label{sub:without}
In the $d$-colour case, it is natural to consider the case of ``drawing without replacement''.
This model is equivalent to allowing the diagonal coefficients of the replacement matrix to be equal to~$-1$. 

To allow drawing without replacement in a MVPP, 
we need to consider again atomic measures since when a measure has no atom, 
the contribution of the weight of the drawing ball to the total mass is zero, 
and then removing it or not does not change anything. 
We decline a variation of our model: the $\kappa$-discrete measure-valued \Pol processes. 

In the $\kappa$-discrete model, for all $x\in\PP$, 
the mass $\mathcal M_n(\{x\})$ of any $x\in \PP$ is a multiple $\kappa$ for some integer $\kappa\geq 2$.
Removing a ball with colour~$x$ corresponds to subtracting $(1/\kappa)\delta_x$ from $\mathcal M_n$, 
which means that $1/\kappa$ corresponds to the weight of a ball. 
In order for the composition measures $\mathcal M_n$ to stay non-negative,
we need to assume that
the initial urn composition~$\mathcal M_0$ 
and the replacement measures (i.e.\ the $(\mathcal R_x, x\in \PP)$) 
are sums of weighted Dirac measures, 
each weight being a multiple of $1/\kappa$. 
This setting corresponds to the generalisation of \Pol urn process ``without replacement'' 
to the measure-valued case.

\bigskip
{\bf Definition of $\kappa$-discrete MVPPs -- }
A $(\mathcal M_0, (\wt{\mathcal R}_x, x\in \PP))$-MVPP is said to be $\kappa$-discrete
if the finite non-negative measure $\mathcal M_0$ can be written under the form 
$\mathcal M_0= (1/\kappa)\sum_{y \in \PP} w_y \delta_y$ where the weights $w_y$'s are non-negative integers, 
all of them being~0 but a finite number, and if for any $x\in \PP$,
\ben\label{eq:wtR}
\wt{\mathcal R}_x = -\frac1\kappa\,\delta_x + \mathcal R_x\een
where
\ben\label{eq:ma}
\mathcal R_x= \frac1\kappa\, \sum_{y \in \PP} r_{x,y}\,\delta_y
\een 
where the $r_{x,y}$'s are non-negative integers all of them being 0 but a finite number. 
In other words, the sequence of integers $(r_{x,y})_{x,y\in\PP}$ is the equivalent for $\kappa$-discrete MVPPs 
of the replacement matrix.
We still assume that $\RR_x(\PP)=1$ for all $x \in \PP$, that is
\ben\label{eq:nb}
\sum_{y\in\PP}r_{x,y}=\kappa.\een

\begin{thm} \label{th:main-2}
Assume that $(\mathcal M_0, (\wt{\mathcal R}_x, x\in \PP))$ is a $\kappa$-discrete MVPP, for some $\kappa\in\{2,3,\ldots\}$. Assume moreover that hypotheses~$(a)$ and~$(b)$ of Theorem~\ref{th:main} hold for the Markov chain with kernel 
$K(z,A)=\mathcal R_z(A)$.
Under these hypotheses, for any finite measure~$\mathcal M_0$ such that $0<\mathcal M_0(\PP)<+\infty$,
we have 
\ben 
\Theta_{a(\beta \log n),\,b(\beta\log n)}\l(n^{-1}{\MM_n}\r) \proba  {\nu}
\een
for the topology of weak convergence on ${\mathcal M}(\PP)$, 
where 
\ben
\beta=1+\frac1{\kappa-1},
\een
and 
{where $\nu$ is}
the distribution of $\Gamma g(\Lambda)+f(\Lambda)$ where $\Lambda$ is a standard Gaussian random variable, and 
$\Gamma\sim \gamma$ independent of~$\Lambda$.
\end{thm}

\begin{rem} \label{rem:cond}
More general models of drawing without replacement can be defined since a weaker tenable condition can be defined: 
what is needed is that for each colour $x$, $\widetilde{\RR}_x(\{x\})$ is a divisor of $\widetilde{\RR}_y(\{x\})$ for all $y\in\PP$, 
when $\widetilde{\RR}_x(\{x\})<0$, 
and there are no condition when $\widetilde{\RR}_x(\{x\})=0$ on $\widetilde{\RR}_y(\{x\})$.
We do not go further in this direction.
\end{rem}

\subsection{Examples and open problems}
\subsubsection{Examples of convergent MVPPs}
In Section~\ref{sub:as}, we discussed two particular examples for which one has strong convergence of the renormalised composition random measure: 
the $d$-colour case and the branching random walk case.
In most other cases, we are unable to prove strong convergence but can still apply Theorem~\ref{th:main} to get convergence in probability;
we now give examples of such cases.

\bigskip
{\bf Homogeneous heavy-tailed random walks -- }
Let $\Delta$ be a random variable on~$\PP$ and let $\mathcal M_n$ 
be the MVPP defined by the replacement measures $\mathcal R_x$ being the law of $x + \Delta$ 
for all $x\in\PP$.
We have already treated the case when $\Delta$ has finite mean and finite variance (see Theorem~\ref{theo:BRW}), but other cases also fall in our framework: the asymptotic behaviour of a random walks is a well-studied topic, in $\mathbb{R}$ but also on much more general Polish spaces ($\mathbb{R}^d$, groups, Cayley graphs, etc.). 
If such a random walk converges (after rescaling) to a limit distribution, 
then it falls in our setting. 

\medskip
{\it The stable case --}
If $\Delta$ is a real random variable having a finite mean $m$ and such that, when $u$ tends to infinity, {$\mathbb P(\Delta\geq u) \sim u^{-\alpha}\ell(u)$ with $\alpha < 2$ and where $\ell$ is regularly varying at infinity}. 
Then, the underlying Markov chain  $(W_n)_{n\geq 0}$ is $(n^{\nicefrac1\alpha}, b(n))$ ergodic with $b(n) = 0$ if $\alpha<1$ and $mn$ otherwise, and its limit law $\gamma$ is $\alpha$-stable. 
In both cases ($\alpha<1$ and $1\leq \alpha <2$), we have $f(x) = 0$, and $g(x) = 1$ and thus, in view of Theorem~\ref{th:main},
\ben
\Theta_{\log^{\nicefrac1\alpha} n,\,0}\l(n^{-1}{\MM_n}\r) \to {\gamma}
\een
in probability when $n$ tends to infinity.

The proof of Theorem~\ref{th:main} relies on the analysis in this case of a branching random walk built on the random recursive trees; 
this result appears to be very similar to that of Fekete~\cite{fekete} where the underlying tree is the binary search tree (Remark \ref{rem:bstvsmvpp} below explains why branching random walks indexed by binary search trees and random recursive trees are very similar objects).

\bigskip
{\bf Your favourite ergodic Markov chain -- }
The philosophy behind Theorem~\ref{th:main} is that
any measure-valued \Pol process is associated to an ergodic Markov chain.
Thus, providing examples of MVPPs to which our result applies is equivalent to providing examples of ergodic Markov chains. 
One may then illustrate our Theorem \ref{th:main} by choosing in the literature a nice Markov chain that converges in distribution, 
for example: the $M/M/\infty$ queue. 
One among many examples is the $M/M/\infty$ queue defined for two positive parameters~$\lambda$ and~$\mu$. 
The Markov chain takes values in $\mathbb N$ and the transition probability are given by
\[p_{x,x+1} = \frac{\lambda}{\lambda+x\mu} \quad\text{ and }\quad p_{x,x-1}= \frac{x\mu}{\lambda+x\mu},\]
for all $x\geq 1$, and $p_{0,1} = 1$.
It is well known that this Markov chain is ergodic and that its stationary distribution is given by
\[\gamma(x) = \Big(\frac{\lambda}{\mu}\Big)^{\!x} \frac{e^{-\nicefrac{\lambda}{\mu}}}{x!} \quad (\forall x\geq 0).\]
Thus, the MVPP $(\mathcal M_n)_{n\geq 0}$ on $\mathbb N$ of replacement measures
\[\mathcal R_x = \frac{\lambda}{x\mu+\lambda}\,\delta_{x+1} + \frac{x\mu}{x\mu+\lambda}\,\delta_{x-1} \quad(\forall x\geq 1),\]
and $\mathcal R_0 = \delta_1$ converges in probability to {$\gamma$}.

\vspace{\baselineskip}
We now want to discuss two extensions we can foresee to this work, but that we have so far not thoroughly investigated.

\subsubsection{Open problem: Random replacement matrices}

In this article, we consider deterministic replacement measures. 
In view of the finite-case literature (see Janson~\cite{Janson2004177}), 
it would be natural to investigate random replacement measures $\RR_x$. 
This model is defined using a family $(\nu_x, x\in \PP)$,
where $\nu_x$ is a probability measure on the set of probability measures on ${\PP}$;  
when the colour~$x$ is drawn for the $k^{\text{th}}$ time, 
we add the measure $\RR_x^{\sss (k)}$ in the urn, 
where $(\RR_x^{\sss (j)}, 1\leq j \leq k)$ are i.d.d.\ taken under $\nu_x$.
We might expect that, for some reasonable assumptions on the deviations of $\RR_x$ around its mean, 
some analogous of Theorem~\ref{th:main} should hold; 
However, we did not investigate this further.

\subsubsection{Open problem: Starting with infinitely many balls}

In the case of a $d$-colour P\'olya's urn (under the assumptions described in the introduction), the total number of balls in the urn is at all times finite, but goes to infinity.
As a mean to understand the ``stationary'' behaviour of the P\'olya's urn at infinity, 
it is natural to
try and define a \Pol urn process 
with an infinite number of balls in the urn (or an infinite mass) at all times. 

It is not possible to define a discrete-time \Pol urn process in this setting 
since choosing a ball uniformly is not possible 
(the measure ${\sf Nor}({\cal M}_0)$ would not be defined). 
However, passing to the continuous-time setting and assuming that at time~$0$ 
the urn contains an infinite number of of balls indexed by the positive integers 
is a way to properly define this process.

Denote by $X_i$ the colour of the $i^{\text{th}}$ ball in the urn at time~$0$ and 
assume that for all colour $c\in\mathbb N$,
\ben\label{eq:rght}
\rho(c):=\lim_{n\to\infty} \frac{\sum_{i=1}^{n} \indi_{X_i = c}}{n}\een
exists, or, more generally (without assuming the countability of the colour space), 
assume that 
\ben\label{eq:rght2}
\rho:=\frac{\sum_{i=1}^{n} \delta_{X_i}}{n}
\een
exists in the space ${\cal M}(\PP)$.

Then equip each of the balls with a clock that rings after an exponentially-distributed random time of parameter~1.
When a clock rings, the associated ball is drawn from the urn and the replacement rule applies.
We assume again that $\RR_y(\PP)=1$ for any $y$ (balance hypothesis). 
The newly added balls/measures are added at the same position as the triggering ball. 
Denote by $\rho(t)$ the limit distribution of ball colours at time $t$, 
limit taken in the sense of~\eref{eq:rght} or~\eref{eq:rght2}.
We may expect that $\rho(t)$ exists
(since it is the sum of the limit measures associated with each lattice point, 
normalised by their total weights),
is deterministic (conditionally on $\rho$), 
and that, for any~$t$
\[\rho(t) \dd \gamma,\]
in the set of probability measures over $\PP$, 
for $\gamma$ defined in Theorem \ref{th:main};
however, we did not investigate this further.

\subsection{Plan of the paper}
In Section~\ref{sec:BMC} we introduce the notion of branching Markov chains (BMC) 
and show how one can couple the measure-valued \Pol process with a branching Markov chain 
on the random recursive tree; 
this section also contains the definition of the random recursive tree and the binary search trees and the statements and proofs of several results about those trees that are then useful when proving the main result.

Section~\ref{sec:main} contains the proof of Theorem~\ref{th:main}. Section~\ref{sec:theo:BRW} contains the proof of Theorem~\ref{theo:BRW} and finally, Section~\ref{sec:without_replacement} treats the without-replacement case and contains the proof of Theorem~\ref{th:main-2}.

\section{Branching Markov chains}\label{sec:BMC}
In this section, we show how to couple the measure-valued \Pol process (MVPP) 
with a branching Markov chain (BMC) on the random recursive tree, 
or equivalently on the binary search tree. 
We also state here some preliminary results about BMCs 
which will be useful when proving our main results.

\subsection{Random recursive tree and binary search tree}
\label{sec:RRDT}

First, consider ${\bf X}=\{\varnothing \}\cup \bigcup_{n\geq 1} \mathbb{N}^n$ and ${\bf X}_2=\{\varnothing \}\cup \bigcup_{n\geq 1} \{0,1\}^n$ the set of finite words on, respectively, the alphabet $\mathbb{N}=\{0,1,2,\ldots\}$ and $\{0,1\}$, where $\varnothing$ is the empty sequence.
We denote by $uw$ the concatenation of the words $u$ and $w$, so that for some letters $a_1,\dots,a_h\in \N$, $a_1\cdots a_h$ is a word with $h$ letters.

\begin{itemize}
\item A {\bf planar tree} $T$ is defined as a subset of ${\bf X}$, containing $\varnothing$ (the root), and which satisfies the two following properties:
\begin{itemize}
\item[--] if $a_1\cdots a_k\in T$ for some $k\geq 1$ then $a_1\cdots a_{k-1}\in T$,
\item[--] if $a_1\cdots a_k\in T$, for any $0\leq j \leq a_{k}$, $a_1\cdots a_{k-1}j\in T$.
\end{itemize}
\end{itemize}
The elements $w$ of $T$ are called nodes, 
and the number of letters in $w$ is denoted $|w|$ -- 
it corresponds to the depth of the node $w$ in the tree. 
Any word~$v$ prefix of~$w$ is called an ancestor of $w$ 
(we write  $v\preceq w$ or $v\prec w$ for the strict property);
 by definition, if $w$ is a node of $T$, then all its ancestors are also in~$T$. 
The siblings of $w=a_1\cdots a_k$ are the elements of the form
$a_1 \cdots a_{k-1}j \in \mathbb{N}^k\cap T$. 
The second condition ensures that the names of the children of any node $w$ are the words $w0$, $w1$, $\ldots, wc$ where $c+1$ is the number of children of $w$. A node in $T$ with no child is called a leaf.

Finally the lexicographical order on ${\bf X}$ induces a total order on every tree.
\begin{itemize}
\item A {\bf complete binary tree} is a planar tree whose nodes belongs to ${\bf X}_2$ 
(in other words, all nodes have 0 or 2 children). 
Nodes with two children are called internal nodes, the other ones are the leaves. 
\item An {\bf incomplete binary tree} is the set of internal nodes of a complete binary trees 
(and it is then not a planar tree, in general, since a node $u$ may have only one child $u1$ without $u0$ being a node of the tree).
In any case, $w0$ is called the left child of $w$, and $w1$, the right one. 
\end{itemize}
Denote by $\Trees_n$, $\IBT_n$ and $\CBT_{2n+1}$ the set of planar trees with~$n$ nodes, 
the set of incomplete binary trees with~$n$ nodes, 
and the set of complete binary trees nodes with~$2n+1$ nodes.

A bijection $g$ between $\CBT_{2n+1}$ and $\IBT_n$ can be 
described as follows:
\begin{itemize}
\item from $T \in \CBT_{2n+1}$, take simply $g(T)$ as the set of internal nodes of $T$,
\item now conversely, take $t$ in $\IBT_n$ and construct ${\sf Complete}(t)=g^{-1}(t)\in \CBT_{2n+1}$ as 
\ben\label{eq:complete}
{\sf Complete}(t)=\{ u0, u \in t\}\cup \{u1,u\in t\}.
\een 
In words, add two children to the leaves of~$t$, and if a node~$u$ has only one child, add the second one.
\end{itemize}

{\bf A rooted recursive tree} with $n+1$ nodes (for some $n\geq 0)$ is a pair $(T,\ell)$ 
where $T\in \Trees_{n+1}$, and $\ell:T\to \{0, \ldots, n\}$ is a bijective labelling of the nodes of $T$, 
such that $\ell$ is increasing on $T$ for the lexicographical order on $T$. 
In other words, $\ell$ increases along the branches starting at the root, 
and along the siblings of each node.

Denote by $\Rec_{n+1}$ the set of rooted recursive trees with $n+1$ nodes.

\bigskip
{\bf The random recursive tree} $(\RRT_n, n\geq 0)$ is a Markov chain described as follows:
\begin{itemize}
\item $\RRT_0=(T_0,L_0)$, where $T_0$ is the tree reduced to its root $\varnothing$, with label $L_0(\varnothing)=0$;
\item assume that $\RRT_n=(T_n,L_n)$ has been built, choose a node $u$ uniformly at random among the~$n$ nodes of~$T_n$. 
Let $T_{n+1} = T_n\cup \{uc\}$, where~$c$ is the smallest integer such that~$uc\not\in~T_n$; 
the labelling $L_{n+1}$ of $T_{n+1}$ coincides with $L_n$ on $T_n$, and we set $L_{n+1}(uc)=n+1$.
\end{itemize}

\bigskip
{\bf The binary search tree} (BST) is a data structure used in computer science to store 
and retrieved data efficiently. It has been deeply studied by many authors. 
The BST associated to a sequence $(x_i, 1\leq i \leq n)$ of distinct elements 
of a totally ordered set (the order being denoted $<$) is a labelled incomplete binary tree $(t,\ell')$ defined recursively as follows. 
At time 1, the tree $t_0$ is reduced to the root~$\varnothing$  (i.e.\ $t_0=\{\varnothing\}$),
which is labelled $\ell'(\varnothing)=x_1$.

To insert a value $x$ in a tree $t$, do the following:
\begin{itemize}
\item if the tree $t$ is empty, create a node, and assign to this node the label $x$. 
\item if the tree is not empty, compare $x$ with the label  $\ell'(r)$ of the root $r$ of $t$. 
If $x>\ell'(r)$ then insert $x$ in the subtree of $t$ rooted at $r0$ else in the subtree of $t$ rooted at $r1$ where $r0$ and $r1$ are the left and right children of $r$.
\end{itemize}

Eventually, the binary search tree associated with $x_1, \ldots, x_n$ is the labelled incomplete binary tree $(t_n,\ell'_n)$ 
with $n$ nodes labelled by $x_1, \ldots, x_n$ obtained by the successive insertions of $x_1, \ldots, x_n$.

\medskip
The random binary search trees under the {\bf permutation model} 
is the pair $(T_n,L'_n)$ associated to the sequence of data $(U_1, U_2, \ldots)$ where the $U_i$ are i.i.d.\ uniformly distributed in $[0,1]$. 
Under this distribution, for all integers~$k$, the sequence $(U_1, \ldots, U_k)$ is exchangeable, 
and thus the (random) permutation $\sigma$ verifying $U_{\sigma(1)}<\cdots<U_{\sigma(k)}$, 
is uniformly distributed on the set of permutations of $\{1, \ldots, k\}$. 
Using an infinite sequence $(U_i, i\geq 1)$ allows one to build 
a sequence of binary trees $((T_n,L'_n), n\geq 1)$. 

The pair $(T_n, L'_n)$ is denoted by $\ov{\BST}_n$ and called the {\bf enriched random binary search tree}.
The first marginal $T_n$ is denoted by $\BST_n$ and called the random binary search tree.
On many occasions, working with  $\ov{\BST}_n$ is a convenient tool to prove results about  $\BST_n$ 
(as for example Lemma \ref{lem:rep}).
We state here a well known fact:
\begin{lem} 
Under the permutation model, $(\BST_n,n\geq 0)$ is the Markov chain defined as follows: $\BST_0=\{\varnothing\}$; 
and for all $n\geq 0$, to build $\BST_{n+1}$ from $\BST_n$, choose a node $u$ uniformly among the leaves of ${\sf Complete}(\BST_n)$, 
and set $\BST_{n+1}=\BST_n\cup\{u\}$.
\end{lem}

In our framework, we will see that the random recursive tree naturally arises in the study of MVPPs. But thanks to the permutation model, the binary search tree is easier to study. We will therefore prove results on the binary search tree and then deduce their counterparts on the random recursive tree via the rotation correspondence, which is a mapping from the set of planar trees onto the set of incomplete binary trees. 

\vspace{\baselineskip}
{\bf The rotation correspondence} is a map $\Psi$  from  $\Trees_{n+1}$ onto $\IBT_n$ (see Figure \ref{fig:rota}). 
\begin{figure}[h]
\centerline{\includegraphics[width = 12 cm]{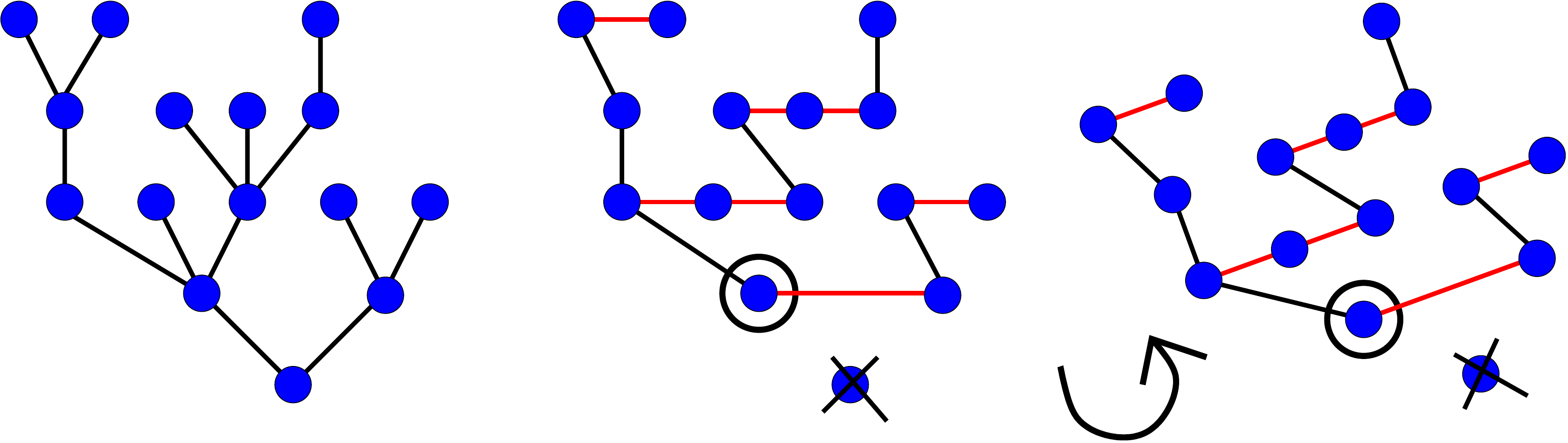}}
\caption{\label{fig:rota} The correspondence by rotation.}
\end{figure} 
The map $\Psi$ is defined at the level of nodes, that is the image of a node $u\in t$ (for a tree $t$) does not depend on $t$, 
but only on~$u$. We denote by $\Psi(u)$ the image of node~$u$ and by $\Psi(t) = \{\Psi(u), u \in t\}$. 

Take a tree $t \in \Trees_{n+1}$ for some $n\geq 1$. 
The tree $\Psi(t)$ is defined as follows (see Figure \ref{fig:rota}):
\begin{itemize}
\item by a matter of size, $t$ contains the node $u=1$; set $\Psi(1)=\varnothing$;
\item assume now that the image $\Psi(t')$ of a subtree $t'$ of $t$ (rooted at $\varnothing$) has been defined. Take a node $v$ in $t \setminus t'$  which is a child of a node $u$ in $t'$:
\begin{itemize} 
\item if $v$ is a leftmost child of node $u$, then set $\Psi(v)=\Psi(u)0$, meaning that the relation parent-leftmost child, is preserved,
\item if $v=a_1 \cdots a_k$ is not the leftmost child of $u$, then $v'=a_1 \cdots a_{k-1}(a_k-1)$ is the left sibling of~$v$. Set  $\Psi(v)=\Psi(v')1$, meaning that the relation sibling-next sibling is transformed into the relation parent-right sibling.
\end{itemize}
\end{itemize}
The following result is classical:
\begin{prop}
For any $n\geq 0$, the rotation correspondence $\Psi$ is a bijection between $\Trees_{n+1}$ and $\IBT_n$.
\end{prop}

The following definitions and lemmas will be useful when translating information on the topology of the binary search tree  into information on the topology of the random recursive tree.
\begin{df} For any two nodes $u_1$ and $u_2$ in a tree, we denote by $u_1 \wedge u_2$ their deepest common ancestor, being their longest common prefix.
For any word $u\in {\bf X}_2$, we define the left-depth $|u|_\ell$ of $u$ as the numbers of $0$-bits it contains.
\end{df}

The rotation correspondence has the following property:
\begin{lem}
$(i)$ For any integer $n$, for any tree $t\in \Trees_{n+1}$ and any node $u\in t$, we have
$|u| = |\Psi(u)|_\ell + 1$.\\
$(ii)$ For any planar tree $t$, and any nodes $u_1, u_2 \in t$, $\Psi(u_1 \wedge u_2)$ is the longest prefix $\omega$ of $\Psi(u_1) \wedge \Psi(u_2)$ such that $|\omega|_\ell < |\Psi(u_1) \wedge \Psi(u_2)|_\ell$.\\ 
$(iii)$ In particular, $|u_1 \wedge u_2| = |\Psi(u_1) \wedge \Psi(u_2)|_\ell$.
\end{lem}

Notice that $(iii)$ follows from $(i)$ and $(ii)$ since $|\Psi(u_1\land u_2)|_\ell = |\Psi(u_1)\land \Psi(u_2)|_\ell - 1$ and $|\Psi(u_1\land u_2)|_\ell = |u_1\land u_2|-1$. 

\begin{lem} 
$(i)$ The rotation correspondence $\Psi$ is a bijective map from $\Trees_{n+1}$ onto $\IBT_n$.\\  
$(ii)$ Its inverse, $\Psi^{-1}$ sends $\BST_n$ onto $\RRT_n$.
\end{lem}
\begin{proof}
The first assertion is folklore (see e.g. Marckert \cite{JFM_RSA} and references therein); 
let us focus on the second one.
Under the permutation model, 
the dynamics of the sequence $(\BST_n, n\geq 1)$ is simple: 
First, $\BST_1$ is reduced to the root.
Now, assume that $\BST_n$ has been defined and is an incomplete binary tree with $n$ nodes.
Let $L_{n}$ be the set of leaves of ${\sf Complete }(\BST_n)$.
It is easy to see that $L_n$ has $n+1$ elements, and that $\BST_{n+1}$ is obtained from $\BST_n$ 
by adding a uniform element of $L_n$. 
Observing the effect of this insertion on $\Psi^{-1}(\BST_n)$, 
one sees that this corresponds to the addition of a child with label $n+1$ as last child of a node 
chosen uniformly at random among the nodes of $\Psi^{-1}(\BST_n)$. 
In other words, the image of the dynamics of the binary search tree 
through the rotation correspondence is the dynamics of random recursive trees $\RRT_n$.
\end{proof}

\paragraph{About the sizes of subtrees in BST.} 

Again the content of this paragraph is well known, and we give explanations principally for the sake of completeness (see e.g. Devroye \& Reed~\cite{De-Reed}, Broutin \& Devroye~\cite{BD}, Chauvin \& al~\cite{chauvin2005} for examples of use of this method). 

We focus here on $\ov{\BST}_n$, the enriched binary search tree
associated a sequence of uniform random variables $(U_i,i\geq 0)$. 
By construction, $U_1$ is inserted to the root $\varnothing$, 
then the $U_i$'s that are smaller than $U_1$ will be inserted in the subtree rooted at~$0$ 
and the ones larger than $U_1$ will be inserted in the subtree rooted at~$1$. 
For $i\in\{0,1\}$, we denote by $\ov{\BST}_n^{(i)}$ the subtree of $\ov{\BST}_n$ rooted at $i$ (being one of the two children of~$\varnothing$). Further, for any node $u$, we let $\ov{\BST}_n^{(u)}$ the subtree of $\ov{\BST}_n$ rooted at $u$.
We denote by~$\pi(\ov{\BST}_n)$ the first coordinate of the pair $\ov{\BST_n}$ 
(it is distributed as $\BST_n$, but we need to keep the overline to denote the enriched model).
\begin{lem} \label{lem:rep}
$(i)$ Conditionally on $U_1$, $|{\BST}_n^{(0)}|$ is binomial $(n-1,U_1)$, and conditionally on $|\BST_n^{(0)}|=k$,
$\pi(\ov{\BST}_n^{(0)})$ and $\pi(\ov{\BST}_n^{(1)})$ are independent and distributed as $\BST_k$ and $\BST_{n-1-k}$.\\
$(ii)$ We have 
 \ben\label{eq:limits}
 n^{-1}\l(|\ov{\BST}_n^{(0)}|,|\ov{\BST}_n^{(1)}|\r) &\as& (U_1,1-U_1), 
 \een
$(iii)$ Set a labelling of the complete binary tree ${\bf X}_2=\{\varnothing \}\cup \bigcup_{n\geq 1} \{0,1\}^n$, by choosing a uniform random variable per node $(V_u, u \in {\bf X}_2)$ and by labelling $u0$ by $W_{u0}=V_u$ and $W_{u1}=1-V_u$ (the root $\varnothing$ is labelled by $W_\varnothing =1$). We have, for all finite subset $F$ of ${\bf X}_2$,
\ben\label{eq:limits2}
n^{-1}\l(|\ov{\BST}_n^{(u)}|, u \in F\r)\as \l( \prod_{z \preceq u} W_{z}, u \in F\r). 
\een
\end{lem}
\begin{proof} $(i)$ Conditionally on $U_1$, the random variables $U_2, \ldots, U_n$ are i.i.d.\  
and each of them is smaller than $U_1$ with probability $U_1$.
Also, conditionally on $U_i\leq U_1$, the random variable $U_i$ is uniformly distributed on $[0, U_1]$ (for all $2\leq i\leq n$).
Therefore, conditionally on $|\ov{\BST}_n^{(0)}|=k$, 
$\pi(\ov{\BST}_n^{(0)})$ is distributed as~$\BST_k$.

$(ii)$ is proved by the exact same argument using additionally the strong law of large number.

$(iii)$ First note that, since $F$ is a finite subset of ${\bf X}_2$, for any node $u\in F$, $|\BST_n^{(u)}| \to \infty$, when $n$ tends to infinity. Let $u\in F$ and denote by $v$ its parent. From~$(i)$, we know that, conditionally on its size, the subtree rooted at $v$ is a random binary search tree under the permutation model. Since the size of the subtree rooted at $v$ goes to infinity with $n$, we can apply~$(ii)$ and get that the size of the tree rooted at~$u$ divided by the size of the tree rooted at the parent of $u$ is asymptotically distributed as $W_u$ (by definition of the~$W_z$'s). The same argument can be done recursively for the parent of $u$, and all its ancestors till the root, which gives the stated result.
\end{proof}

We end this section by a lemma whose proof is straightforward. 
Let $t$ be a binary tree and $u$ a node of~$t$. 
Denote by $\TR_u(t)$ the tree obtained by exchanging the two subtrees of $t$ rooted at $u0$ and $u1$. 
Formally $\TR_u(t)$ is obtained by replacing all words (nodes) $u0w$ in $t$ (resp.\ $u1w$) by $u1w$ (resp.\ $u0w$). 
If $u\not\in t$, let $\TR_u(t)=t$.  
\begin{lem}\label{eq:grg} 
Let $\BST_n$ be the random binary search tree under the permutation model, for some $n\geq 0$. 
\begin{enumerate}[(i)]
\item For any node $u$,
$\TR_u(\BST_n)\eqd \BST_n$;
\item let $w=w_1 \cdots w_{|w|}$ be a node chosen uniformly in $\BST_n$, then the letters $w_i$'s are i.i.d.\ random variables, uniformly distributed on~$\{0,1\}$.
\end{enumerate}
\end{lem}
\begin{proof} $(i)$ follows by symmetry of the construction of the random binary search tree. $(ii)$ is a straightforward consequence of $(i)$.
\end{proof}

\subsection{Branching Markov chain}
\label{sec:MBC}
Branching random walks are classical objects in probability theory. They are random walks indexed by a rooted tree: with each node $u$ of a tree $t$ is associated a random variable $\Delta_u$, the family $(\Delta_u,u \in t)$ being i.i.d., and, by convention, we set $\Delta_\varnothing=0$ where $\varnothing$ is the root. Now, the branching random walk is the pair $(t,(X_u,u\in t))$ where $X_u = \sum_{v \preceq u } \Delta_v$, 
so that along a branch $X_v$ evolves as a random walk. 
The name branching random walk comes from the dependence structure: for any two nodes $(u,v)\in t$,
\ben
(X_u, X_v)\eqd (Z_{|u \wedge v|},Z_{|u\wedge v|}) + (Z'_{|u|-|u \wedge v|},Z''_{|v|-|u\wedge v|})
\een
where in the right hand side $Z,Z',Z''$ denote independent random walks starting at 0. 
At the core of our work lies the notion of branching Markov chains, 
which have been considered 
in Bandyopadhyay and Thacker \cite{BT1}, also in the context of \Pol urn processes. 
Here we extend a bit their definition, and go further in the analysis to prove Theorem \ref{th:main}.
\begin{df} 
A branching Markov chain (BMC) with initial position $X_\varnothing$ and family of kernels $(\ov{K}^c,c\geq~1)$
is a stochastic process $X(t)=(X_u, u \in t)$ indexed by a tree $t$ with the following properties:
\begin{itemize}
\item the variables attached to the children of the root $(X_j, 0\leq j \leq c_{\varnothing}-1)$, 
are independent and distributed as $\ov{K}^{c_{\varnothing}}(X_\varnothing,\,\cdot\,)$; in other words  for any Borel sets $(B_0, \ldots, B_{c_\varnothing -1})$,
\ben
\P((X_{0}, \ldots, X_{c_{\varnothing}-1}) \in B_0\times \cdots \times B_{c_\varnothing-1} | X_\varnothing) =\ov{K}^{c_\varnothing}(X_\varnothing, B_0\times \cdots \times B_{c_\varnothing-1}).
\een
\item conditionally on $(X_j, 0\leq j \leq c_{\varnothing}-1)$, 
the families $X(t_j), 0\leq j \leq c_{\varnothing}-1$ attached to subtrees $t_j$ rooted at the children of the root, are independent BMCs with respective initial positions $X_j$.
\end{itemize}
\end{df}

We will call \it $K$-simple branching Markov chain (SBMC) with kernel $K$\rm, 
a BMC such that, for all $c\in\{1, 2, \ldots\}$ and $x\in \PP$, 
for all Borel sets $B_0, \ldots, B_{c-1}$,
\ben
\ov{K}^c(x,  B_0\times \cdots \times B_{c_1})= K(x,B_0)\times\cdots\times K(x,B_{c-1}).
\een
In a $K$-SBMC the values associated to siblings are independent conditionally on the value of their parent and~$X_v$ 
evolves on each branch $B_u$ as a Markov chain with initial position $X_{\varnothing}$ and kernel~$K$.
For any two nodes $(u,v)\in t^2$, we have
\ben
(X_u,X_v)\eqd (M_{|u \wedge v|},M_{|u\wedge v|}) + (M'_{|u|-|u \wedge v|}-M_{|u\wedge v|},M''_{|v|-|u\wedge v|}-M_{|u \wedge v|})
\een
where, in the right hand side, $M$ is a Markov chain starting at $X_\varnothing$, and, conditionally on $M_{|u\wedge v|}$, $(M',M'')$ are two independent Markov chains starting at position $M_{|u \wedge v|}$ (all these Markov chains having the same kernel ${K}$).

\subsection{Coupling of the MVPP with a BMC}
We couple (or encode) the sequence $(\MM_n, n\geq 0)$ with a sequence of branching Markov chains $(X(\RRT_n), n\geq0)$ on the random recursive tree.

In Section \ref{sec:MBC}, we defined BMCs on a fixed underlying tree $t$. 
We now need to consider a sequence of BMCs having as sequence of underlying trees the sequence of $(\RRT_n, n\geq 0)$. 
The sequence $(\RRT_n, n\geq 0)$ being a nested sequence of trees, 
we can define a nested sequence of BMCs using these trees, as follows.
First assume that a kernel $K$ and an initial distribution $\mathcal M_0$ verifying ${\cal M}_0(\PP)=1$ are given. 
Let $u_{n+1}$ be the only node in $\RRT_{n+1}\setminus \RRT_n$, 
and let $v$ be its parent in~$\RRT_{n+1}$. 
Conditionally on the labels $(X_u, u \in \RRT_n)$, 
take $X_{u_{n+1}}$ under the distribution $K(X_v,\cdot\,)$. 
This defines a sequence of compatible $K$-SBMC that we denote by $(X(RRT_n),n\geq 0)$.
\begin{lem}\label{lem:coupl_RRT}
Let $(X(\RRT_n),n\geq 0)$ be the sequence of compatible $K$-SBMC defined above, with initial distribution $\mathcal M_0$ such that ${\cal M}_0(\PP)=1$ 
and kernel $K$ defined for any $x \in \PP$, 
\ben K(x,\cdot\,)=\mathcal R_x(\,\cdot\,).\een
Then the process defined for all integers $n$ by
\ben
\mathcal M^{\star}_n = \mathcal M_0 + \sum_{u \in\RRT_n\setminus \{\varnothing\}} \mathcal R_{X_u}\een
satisfies $(\mathcal M_n^\star)_{n\geq 0} = (\mathcal M_n)_{n\geq 0}$ 
where $( \mathcal M_n)_{n\geq 0}$ is the MVPP of initial composition $\mathcal M_0$ 
and replacement measures $(\mathcal R_x)_{x\in\mathcal P}$.
\end{lem}

\begin{proof}
 It suffices to prove that the sequence of measures $({\cal M}_n^{\star}, n\geq 0)$ is a Markov chain, and that it has the same kernel as~$({\cal M}_n, n\geq 0)$ (as well as the same initial distribution but this is straightforward).

 For the first property, recall that, to build $\RRT_{n+1}$ from $\RRT_n$, 
 one chooses uniformly at random a node $v\in\RRT_n$ and adds a new child $u_{n+1}$ to $v$. 
 Therefore, in the branching random walk, 
 the distribution of the new label $X_{u_{n+1}}$ does not depend on the geometry of tree, 
 but only on the already existing labels $(X_u, u \in \RRT_n)$. 
 This ensures the fact that ${\cal M}_n^{\star}$ is a Markov chain.

 For the second property, it suffices to notice that that the only difference 
 between the MVPP and the BMC representation ${\cal M}_0+\sum_{u \in\RRT_n\setminus \{\varnothing\}} \mathcal R_{X_u}$ is that, 
 in this latter, the data (the current values $X_.$) are differently organised. But the measures ${\cal M}_n^\star$ do not depend on this organisation. 
\end{proof}

\begin{cor}\label{thm:rr} 
Let $(\mathcal M_n, n\geq 0)$ be the MVPP of replacement measures $(\mathcal R_x)_{x\in\PP}$ 
with initial measure ${\cal M}_0$ such that ${\cal M}_0(\PP)=1$.
Let $X(\RRT_n)$ be the $K$-SBMC on the random recursive tree of initial distribution $\mathcal M_0$ 
and kernel $K(x,\,\cdot\,) = \mathcal R_x(\,\cdot\,)$ (for all $x\in\PP$).

Let $(\ov{A_n},\ov{B_n})$ be a pair of independent random variables taken under the random probability distribution $n^{-1}{\cal M}_n$. Then the random variable $(\ov{A_n},\ov{B_n})$ has distribution $(\mathcal R_{X_{U_n}}, \mathcal R_{X_{V_n}})$, where $U_n$ and $V_n$ are two uniform and independent nodes in $\RRT_n$.
\end{cor}

\begin{rem} \label{rem:bstvsmvpp}
It is interesting to note that  MVPPs can also be encoded by non-simple BMCs indexed by the BST.
To see this, consider $\CBST_n$ the complete binary search trees (this is the binary tree whose set of internal node is $\BST_n$).
Define a branching random chain having $\CBST_n$ as underlying tree, with initial distribution ${\cal M}_0$ 
and kernel $\ov{K}$ defined as followed: for all measurable sets $\mathcal A$ and $\mathcal B$,
\ben\label{eq:ks}
\ov{K}(x, \mathcal A \times \mathcal B) = \frac{1}2\l(1_{x \in A} \RR_x(B) + 1_{x \in B}\RR_x(A)\r).
\een
In other words: to generate the value $(X_{u0},X_{u1})$ of the children of $X_u$, flip a fair coin:
\begin{itemize}
\item if it is tails, set $X_{u0}=X_u$ and draw $X_{u1}$ according to the kernel $\RR_{X_u}(\,\cdot\,)$;
\item if it is heads, then take $X_{u1}=X_u$ and draw $X_{u0}$ according to the kernel $\RR_{X_u}(\,\cdot\,)$.
\end{itemize}
Then, the process defined for all integers $n$ by 
$\mathcal M^{\bullet}_n = \mathcal M_0 + \sum_{\nu\in Leaves\l(\CBST_n\r)} \mathcal R_{Y_\nu}$ 
is equal in distribution to the MVPP of initial composition $\mathcal M_0$ 
and replacement measures $(\mathcal R_x)_{x\in\mathcal P}$. 
Since $(\CBST_n, n\geq 0)$ is also a sequence of nested trees, 
one may define a compatible sequence of BMCs and check that 
$(\mathcal M_n^\bullet)_{n\geq 0} = (\mathcal M_n)_{n\geq 0}$ in distribution.

To see this  one encodes the evolution of the MVPP by a binary search tree, storing the information at the level of leaves (while in the RRT-case, we work at the level of all nodes). When ``one draws a node $u$'' with value $X_u$, we let it there, and add a child to $u$ with value distributed according to $\RR(X_u,\,\cdot\,)$. 
The same encoding can be realised by, instead, drawing only leaves, and when one draws a leaf $u$ with value $X_u$, they add to this leaf two children $u0$ and $u1$, copy the value of $u$ in $u0$ or $u1$ at random with probability $1/2$ and draw the value of the other child at random according to~$\RR(X_u,\,\cdot\,)$. 
\end{rem}

\subsection{Auxiliary results  on RRT's and BST's}

An important ingredient of our proof of Theorem \ref{th:main} is that we know the depth of a node/two nodes in the random recursive tree and in the random binary search tree: 
\begin{prop}\label{prop:inc-tree-BST} 
Let $U_n$ and $V_n$ be two random uniform and independent nodes taken in $\BST_n$.
\begin{enumerate}[(i)]
\item Asymptotically when $n$ goes to infinity, we have
\ben\label{eq:sefdf}
\l(\frac{|U_n| - 2\log n}{\sqrt{2\log n}},\frac{|V_n| - 2\log n}{\sqrt{2\log n}}, |U_n \wedge V_n|\r) \dd \l(\Lambda_1,\Lambda_2,K_{1/3}\r),
\een 
where the three r.v. are independent, $K_{1/3}\sim {\sf Geometric}(1/3)$,  $\Lambda_1$ and $\Lambda_2$ are $\mathcal N(0,1)$-distributed.
\item Asymptotically when $n$ goes to infinity, we have
\ben
\l(\frac{|U_n|_\ell - \log n}{\sqrt{\log n}},\frac{|V_n|_\ell - \log n}{\sqrt{\log n}}, |U_n \wedge V_n|_\ell\r) \dd \l(\Lambda_1,\Lambda_2,K\r),
\een 
where the three r.v. are independent,  $K\sim {\sf Geometric}(1/2)$,  $\Lambda_1$ and $\Lambda_2$ are $\mathcal N(0,1)$-distributed. 
\end{enumerate}
\end{prop}

As a corollary of this theorem, using the rotation map, we immediately get
\begin{prop}\label{prop:inc-tree} 
Let $U_n$ and $V_n$ be two random uniform and independent nodes taken in  $\RRT_n$.
We have
\ben\label{eq:mult-conv}
\l(\frac{|U_n| - \log n}{\sqrt{\log n}},\frac{|V_n| - \log n}{\sqrt{\log n}}, |U_n \wedge V_n|\r) \dd \l(\Lambda_1,\Lambda_2,K\r),
\een 
where the three r.v. are independent,  $K\sim {\sf Geometric}(1/2)$,  $\Lambda_1$ and $\Lambda_2$ are $\mathcal N(0,1)$-distributed. 
\end{prop}

\begin{rem} The results presented in Propositions \ref{prop:inc-tree-BST}  and \ref{prop:inc-tree} are partially known. The convergence of $|U_n \wedge V_n|$ in the RRT case and binary cases are proved in Kuba \& Wagner \cite{KubaW10}. The asymptotic normal distribution for the depth of a uniform node, are due to Dobrow~\cite{Dobrow96} for the RRT, and to Mahmoud \& Pittel \cite{Mah} for the BST. 

In the propositions stated above, 
we prove joint convergence in distribution, 
which is stronger that the marginal convergence already proved in the literature.

Note that stronger results are known about the profile of these trees (which encodes the number of nodes at each level): see Chauvin \& al. \cite{CDJH,chauvin2005} for the BST and Fuchs \& al. \cite{Fuchs2006} for partial results about the profile of the RRT. But these results do not imply the above propositions.
\end{rem}

\begin{proof}[Proof of Proposition \ref{prop:inc-tree-BST}] 
$(i)$ is a consequence of the third marginal convergence - a result due to Kuba \& Wagner \cite[Theorem 7]{KubaW10} - and of the fact that 
\ben\label{eq:Mah}
\frac{|U_n|-2\log n}{\sqrt{2\log n}} \dd \mathcal N(0,1),
\een 
a result due to Mahmoud \& Pittel \cite{Mah} (see also Devroye \cite{Devroye}). 
To see this, proceed as follows. 
We work with the enriched random binary search tree, which has, in terms of depth of random nodes, the same properties as the random binary search tree. By Lemma \ref{lem:rep}, the vector of the sizes of the subtrees rooted at a depth smaller than $k$ (sorted according to their root's lexicographical order) converges almost surely on the enriched space to a limit which has no entries equal to 0:
\[\l(n^{-1}|\ov{\BST}_n^{(u)}|, (u, |u|\leq k)\r)\as \l( \prod_{z \preceq u} W_{z},  (u, |u|\leq k)\r).\]
On this enriched space, the probability that $U_n\wedge V_n =u$, 
where $u$ is any given word of length $k$ converges to 
\[p_u:=2\l(\prod_{z \preceq u0} W_z\r)\l(\prod_{z \preceq u1} W_z\r)\]
since these terms are the asymptotic proportions of nodes in the subtrees rooted at $u0$ and $u1$.
We thus get that
\ben\label{eq:cocns}
\l(n^{-1}|\ov{\BST}_n^{(U_n\wedge V_n)0}|,n^{-1}|\ov{\BST}_n^{(U_n\wedge V_n)1}|, U_n\wedge V_n\r)\dd (\alpha, \beta, W)
\een for some random variables $\alpha, \beta$ and $W$. 
Moreover $\alpha$ and $\beta$ are almost surely positive.

It remains to describe $(U_n,V_n)$ conditionally to the event 
$\mathcal E:= \{(|\ov{\BST_n}^{(U_n\wedge V_n)0}|,|\ov{\BST_n}^{(U_n\wedge V_n)1}|, U_n\wedge V_n)=(s_1,s_2,W)\}$. 
Conditionally on $\mathcal E$:
\begin{itemize}
\item $\ov{B_1}=\pi(\BST_n^{(U_n\wedge V_n)0})$ and $\ov{B_2}=\pi(\BST_n^{(U_n\wedge V_n)1})$ are independent and are distributed respectively as~$\BST_{s_1}$ and~$\BST_{s_2}$;
\item $U_n$ (resp. $V_n$) is a node taken uniformly at random in $\ov{B}_1$ (resp. $\ov B_2$).
\end{itemize}
Hence, conditionally on $\mathcal E$,
\ben\label{eq:rep}
(|U_n|,|V_n|,U_n\wedge V_n)\eqd \l(|W|+|U^{s_1}|, |W|+|U^{s_2}|,W\r)
\een
where $|U^{s_1}|$ is independent of $|U^{s_2}|$, and $U^s$ is a uniform node in $\BST_s$. 
Now, we can conclude the proof of $(i)$: by Skorokhod representation theorem, the weak convergence stated in \eref{eq:cocns} holds a.s.\ on a certain probability space. By the representation given in \eref{eq:rep} and by \eref{eq:Mah}, it follows that on this space
\ben\label{eq:cogn}
\l(\frac{(|U_n|-|U_n\wedge V_n|) -  2\log (\alpha n)}{\sqrt{\log(\alpha n)}},\frac{(|V_n|-|U_n\wedge V_n|)-2\log (\beta n)}{\sqrt{\log(\beta n)}}, |U_n\wedge V_n|\r) \to (\Lambda_1,\Lambda_2,|W|)  \een
where the three random variables are independent,  $|W|\eqd K_{1/3}$,  $\Lambda_1$ and $\Lambda_2$ are $\mathcal N(0,1)$-distributed. From here, one sees that since $\alpha$ and $\beta$ are almost surely positive, Equation~\eref{eq:cogn} implies Equation~\eref{eq:sefdf}.

$(ii)$ This assertion is in fact a consequence of the first one and of Lemma~\ref{eq:grg}$(ii)$. 
Conditionally on $|U_n\wedge V_n|=k$, since $U_n\wedge W_n$ is by symmetry uniform among the words with $k$ letters on the alphabet $\{0,1\}$, 
$|U_n \wedge V_n|_\ell$ is binomial$(k,1/2)$. It is easy from there to recover that, since $|U_n\wedge V_n|$ is geometric of parameter $1/3$, $|U_n\wedge V_n|_\ell$ is geometric of parameter~$1/2$. It now remains to adapt the rest of the previous proof. Following the steps of the proof of $(i)$, one sees that $|U^{s_1}|_\ell$ and $|U^{s_2}|_\ell$ are independent, and by Lemma~\ref{eq:grg}$(ii)$, $|U^s|_\ell$ is, conditionally to $|U^s|$, binomial $(|U^s|,1/2)$. 
The fact that $(i)$ implies $(ii)$ is a consequence of the following general statement (easy to prove, e.g.\ using the central limit theorem and the Skorokhod representation theorem for $X_n$):

Assume that $(X_n,Y_n)$ is a sequence of random variables, such that:
\begin{enumerate}[(a)]
\item the random variables $(X_n, n\geq 0)$ are almost surely non-negative,
\item the distribution of $Y_n$ conditionally to $X_n$ is a binomial of parameter $(X_n,1/2)$,
\item $(X_n-a_n)/\sqrt{a_n} \dd {\cal N}(0,1)$ (for some diverging sequence $(a_n, n\geq 0)$).
\end{enumerate}
Then $(Y_n-a_n/2)/\sqrt{a_n/2}\dd {\cal N}(0,1)$, when $n$ goes to infinity.
\end{proof}

\section{Proofs of Theorem \ref{th:main}}
\label{sec:main}

\subsection{Preliminary lemma}
\begin{lem}\label{lem:indep} Let $(\nu_n, n\geq 0)$ be a sequence of random probability measures with total mass 1. 
For any integer~$n$, take $(A_n,B_n)$ two independent random variables with common distribution $\nu_n$. 
If
\ben \label{eq:pair}
(A_n,B_n)\dd(A,B)
\een where $(A,B)$ are two independent random variables with a deterministic distribution $\nu$ then $\nu_n\dd \nu$ for the topology of weak convergence in ${\cal M}(\PP)$\footnote{Given a sequence of random variables $(X_n)_{n\geq 0}$ and a random variable $X$, we say that $X_n\dd X$ if the distribution of $X_n$ converges weakly to the distribution of $X$.}.
\end{lem}
\begin{proof} For the sake of completeness we give a proof of this lemma although it is folklore. 
The weak convergence in distribution of a sequence of random measures $(\nu_n)_{n\geq 0}$ on $\PP$ 
to  $\nu$ is equivalent to the convergence 
\ben\label{eq:cvloi}
\int \Phi\, d\nu_n \dd \int \Phi\, d\nu
\een
for any bounded continuous function $\Phi:\PP\to \mathbb{R}$. 
{Since its right term is deterministic, Equation~\eqref{eq:cvloi} follows from}
\ben
\mathbb{E}\l(\int \Phi\,d\nu_n\r)\to \int \Phi\,d\nu~~\textrm{ and }~~ \Var\l(\int \Phi\, d\nu_n\r) \to 0.
\een
The first convergence can be restated under the form $\E(\Phi(A_n))\to \E(\Phi(A))$ which is a consequence of the convergence of the first marginal in \eref{eq:pair}.  Now, 
\ben
\Var\l(\int \Phi \,d\nu_n\r)=\E\l[\bigg(\int \Phi \,d\nu_n\bigg)^2\r]-\E\bigg[\int \Phi \,d \nu_n\bigg]^2
                      = {\sf Cov}(\Phi(A_n),\Phi(B_n)),
\een
and since $\Phi$ is bounded and continuous, \eref{eq:pair} implies that ${\sf Cov}(\Phi(A_n),\Phi(B_n))\dd 0$, which concludes the proof.
\end{proof}

We prove Theorem~\ref{th:main} in two steps, separated in two subsections: we first assume that the initial composition measure $\mathcal M_0$ has total mass one; and then show how the result can be generalised to any initial composition measure.

\subsection{Proof of Theorem~\ref{th:main} when ${\cal M}_0(\PP)=1$}
For any $ n \geq 1$, set $\mu_n:= n^{-1} \Theta_{a(\log(n)),\,b(\log(n))}\l(\MM_n\r)$. 
The sequence $(\mu_n, n\geq 0)$ is a sequence of random probability measures, since each $\mu_n$ has total mass~1 as we have assumed ${\cal M}_0(\PP)=1$. 

In order to apply Lemma~\ref{lem:indep}, we take $U_n$ and $V_n$ two nodes taken independently and uniformly at random in the random recursive tree $\RRT_n$ 
and denote by $\ov{A_n}$ and $\ov{B_n}$ two independent random variables 
of respective distributions $\mathcal R_{X_{U_n}}$ and $\mathcal R_{X_{V_n}}$.
In view of Lemma \ref{lem:indep} and Theorem \ref{thm:rr}$(ii)$, to prove Theorem~\ref{th:main}, 
it suffices to prove that 
\[(A_n,B_n):= \l(\frac{\ov{A_n}-b(\log n)}{a(\log n)}, \frac{\ov{B_n}-b(\log n)}{a(\log n)}\r)\]
converges in distribution towards a pair of independent random variables 
with common distribution that of $\Gamma  g(\Lambda) + f(\Lambda)$ where $\Gamma$ and $\Lambda$ are independent, 
$\Lambda$ is a standard Gaussian random variable, and $\Gamma$ is $\gamma$-distributed.
{This would indeed imply that $\Theta_{a(\log n),\,b(\log n)}(n^{-1}\mathcal M_n)$ converges in distribution to the deterministic measure $\nu$, which, in turn, implies convergence in probability of $\Theta_{a(\log n),\,b(\log n)}(n^{-1}\mathcal M_n)$ to~$\nu$.}

Conditionally on $K_n=U_n \wedge V_n$, the BMC structure implies that
\[(\ov{A_n},\ov{B_n})\eqd (W_{|K_n|},W_{|K_n|}) + (W^{(1)}_{1+|U_n|-|K_n|}-W_{|K_n|}, W^{(2)}_{1+|V_n|-|K_n|}-W_{|K_n|})\]
where $W$ is a Markov chain of kernel $K$ of initial distribution $\mathcal M_0$, 
and $W^{(1)}$ and $W^{(2)}$ are two independent Markov chains of Kernel $K$ and of initial distribution~$\delta_{W_{|K_n|}}$.
By the Skorokhod representation theorem, one can work on a probability space on which the convergence stated in Proposition \ref{prop:inc-tree} is almost sure. On this space  
\be
|U_n|+1 &=& \log n + \Lambda_1 \sqrt{\log n} +\varepsilon_1(n)\\
|V_n|+1 &=& \log n + \Lambda_2 \sqrt{\log n} +\varepsilon_2(n)\\ 
|K_n| & \as & G
\ee
where $\varepsilon_1(n)$ and $\varepsilon_2(n)$ are two random error terms, almost surely negligible with respect to $\sqrt{\log n}$, 
$\Lambda_1$ and $\Lambda_2$ are two independent standard Gaussian random variables, 
and $G$ is a finite (geometric) random variable. 
Notice that this last convergence implies that~$|K_n|$ is eventually constant equal to~$G$ 
for every~$n$ greater than some (random) integer~$n_0$. 
For $n\geq n_0$, we have
\[(\ov{A_n},\ov{B_n})\eqd (W^{(1)}_{1+|U_n|-G}, W^{(2)}_{1+|V_n|-G})\]
where $W^{(1)}$ and $W^{(2)}$ are two independent Markov chains starting at a position $W_G$. 
Since $G$ is fixed, the starting position of $W^{(1)}$ and $W^{(2)}$ is now fixed, 
and the ergodicity hypothesis applies.
In fact, since $W^{(1)}$ and $W^{(2)}$ are independent, 
it suffices to find the limit of $(W^{(1)}_{1+|U_n|-G}-b(\log n))/a(\log n)$ and to observe that this limit is independent from $G$. 
To see this, one may, for example, condition on the value of $W_G$, and assume in the sequel that it is fixed.
Write 
\ben\label{eq:repr}
\frac{W^{(1)}_{{|U_n|+1-G}} - b(\log n)}{a(\log n)} = \frac{W^{(1)}_{{|U_n|+1-G}}-b_{{|U_n|+1-G}}}{a_{{|U_n|+1-G}}} \frac{a_{{|U_n|+1-G}}}{a(\log n)}+ \frac{b_{{|U_n|+1-G}}-b(\log n)}{a(\log n)}\ .
\een
Since $G=o(\sqrt{\log n})$, by assumption $(b)$ of the theorem, we have
\[\frac{b_{{|U_n|+1-G}}-b(\log n)}{a(\log n)} \to f(\Lambda_1)\quad\text{ and }\quad
\frac{a_{{|U_n|+1-G}}}{a(\log n)} \to g(\Lambda_1),\]
where $\Lambda_1$ is independent of~$G$.
By assumption~$(a)$, 
\[\frac{W^{(1)}_{{|U_n|+1-G}}-b_{{|U_n|+1-G}}}{a_{{|U_n|+1-G}}} \dd \Gamma_1,\] 
where $\Gamma_1$ is independent of~$G$ and $\gamma$-distributed.
In conclusion, 
\[\frac{W^{(1)}_{|U_n|+1}-b(\log n)}{a(\log n)} \dd \Gamma_1 g(\Lambda_1)+f(\Lambda_1),\] 
and this variable is independent of~$G$. 
This concludes the proof of Theorem~\ref{th:main} under the assumption that~$\mathcal M_0(\PP) = 1$.

\subsection{Proof of Theorem~\ref{th:main} for general ${\cal M}_0(\PP)$}
\label{sec:generalM0}

To conclude the proof of Theorem~\ref{th:main}, we need to discuss the case when $\mathcal M_0(\PP)\neq 1$.

Assume first that ${\cal M}_0(\PP)=m$ is an integer. In this case, the idea consists in splitting the initial measure into $m$ parts $(\mathcal M_0^{(i)},1\leq i \leq m)$ (that is such that ${\cal M}_0=\sum_{i=1}^{m}{\cal M}_0^{(i)}$ and ${\cal M}_0^{(i)}(\PP)=1$),
each of them having total mass~1. Sampling according to $\mathcal M_0$ is thus equivalent to first choosing a uniform value $i$ in $\{1,\ldots,m\}$ and then sampling according to $\mathcal M_0^{(i)}$.

Consider the forest built as follows: at time zero, the forest is composed of $m$ trees reduced to their roots.
At every discrete time step, one draws a node uniformly at random in the forest, and add a child to this node.
Note that, conditioned on their sizes $(s_1^{(n)}, \ldots, s_m^{(n)})$, each of the $m$ trees of the forest are independent random recursive trees, and then, to get Theorem \ref{th:main} in this setting it suffices to show that the asymptotic sizes of these trees are linear (since this holds for any starting distributions ${\cal M}_0^{(i)}$).
Note that the vector $(s_1^{(n)}, \ldots, s_m^{(n)})$ is the composition vector of a $m$-colour urn process of initial composition vector $(1, \ldots, 1)$ and replacement matrix $\mathtt{Id}_m$. It is known that
\begin{lem}[{see for example~\cite{JK97}}]\label{lem:Dirichlet}
\ben\label{eq:conv}
\l(s_j^{(n)}/{n},1\leq j \leq m\r)\as \l(s_j,1\leq j \leq m\r)\een
and the limit follows the Dirichlet$(1,\ldots,1)$ distribution, implying in particular that $s_j>0$ almost surely for all $1\leq j \leq m$.
\end{lem}

If $m$ is not an integer we can again couple the MVPP with a BMC on a random forest composed of $\lfloor m \rfloor+1$ trees. The random forest is built as follows: at time zero, it is composed of the roots of $\lfloor m\rfloor+1$ trees, the first $\lfloor m\rfloor$ have weight~1 and the last has weight $\{m\}:=m-\lfloor m\rfloor$. At each discrete time step, one picks a node at random in the forest with probability proportional to its weight, and adds a child of weight~1 to this randomly chosen node. 
Note that the $\lfloor m\rfloor$ first trees, conditioned on their size, are random recursive trees, and the last one has a slightly different distribution: we weight its root by $\{m\}:=m-\lfloor m\rfloor$ and each other of its nodes by~1. 

Again, we can conclude if we can prove that under these dynamics the tree sizes are asymptotically linear (see also Remark \ref{rem:tj} below).

Note that the sizes of the first $\lfloor m\rfloor$ trees of the forest have asymptotically a linear size in $n$. This can be seen by comparison with the case when the initial mass is $\lfloor m\rfloor+1$. 
In fact, the last tree also has asymptotic linear size: let us denote by $T$ the first time (in the construction of the random forest) 
that a child is added to the root of the last tree. Note that $T$ is almost surely finite (since at time $n$, the probability is $\{m\}/(n+\{m\})$). 
At time $T$, the first $\lfloor m\rfloor$ trees of the forest contain $\lfloor m\rfloor+T$ nodes (all of weight one), and the last tree contains one node of weight one, which we denote by $\nu$ (plus the root of weight $\{m\}$). Thus, the size of the last subtree is larger than the size of the subtree rooted at $\nu$, which we denote by $s_{\nu}(n)$. 
Again, by Lemma~\ref{lem:Dirichlet}, conditionally on~$T$, $n^{-1}{s_\nu(n)}$ converges almost surely to a Beta-distributed random variable $b$ of parameter $(1, T)$, which implies that $s^{(n)}_{m+1}$, the size of the last subtree, satisfies a.s.
\[b \leq \liminf_n s^{(n)}_{m+1}/n \leq \liminf_n s^{(n)}_{m+1}/n \leq 1.\]
\begin{rem}\label{rem:tj}
Given its size, the last subtree is not distributed as a random recursive tree because of the weight of the root.
Luckily, the subtrees of the root, given their sizes are distributed as random recursive trees.
Moreover, if we compare with the case when the initial mass is~1, 
the subtrees of the root are less numerous and larger than in the random recursive tree case.
\end{rem}

\section{Proof of Theorem~\ref{theo:BRW}}
\label{sec:theo:BRW}

We first prove the result in dimension $d=1$.

\subsection{One-dimensional case}

Denote by  $(X_1, \ldots, X_k)$ the first $k$ values 
of the branching Markov chain at the 1st, 2nd, 3rd... nodes, in their order of appearance in the tree. Consider the map $Z_n:\C\to \C$ defined by
\ben
Z_n(x)=\prod_{j=1}^n \l(\frac{j-1}{j}+\frac{x}{j}\r).
\een
Notice that $Z_n(1)=1$. For all $\theta\in\mathbb C$, set
\ben\label{eq:Fnt}
F_n(\theta)= Z_n(e^{-im\theta})\sum_{k=1}^{n+1} \frac{e^{i\theta X_k}}{n+1}
\een
a rescaled version 
of the empirical Fourier transform of the random probability measure 
\ben
\rho_{n+1}:=\sum_{k=1}^{n+1} \frac{\delta_{X_k}}{n+1}.
\een
Notice that 
\ben
\frac1{n+1}\mathcal M_{n+1}= \int R_{x}(\,\cdot\,) d\rho_{n+1}(x). 
\een
Hence ${\cal M}_{n+1}$ is the distribution of $X_{U_n}+\Delta$ where $U_n$ is uniform in $\{1, \ldots, n+1\}$. 
Now, for any sequences $(a(n), n\geq 0)$ and $(b(n), n\geq 0)$ such that $a(n)\to +\infty$, and any distribution $\rho$, we have that
\ben\label{eq:sgegz}
\Theta_{a(n),b(n)}(\rho_n)\as \rho \quad\text{implies}\quad \Theta_{a(n),b(n)}\l(n^{-1}{\cal M}_n\r)\as \rho.
\een
It is thus enough to prove that $\Theta_{a(n),b(n)}(\rho_n)\as \rho$.
Let 
\[T_n(\theta)= \frac{ F_n(\theta)}{\E(F_n(\theta))}\]
its renormalised version (note that $F_0(\theta)=T_0(\theta)=1$). The case $m=0$ corresponds to the case $Z_n(1)=1$. 
In view of the dynamics of the MVPP, we have the following recursion: for all $n\geq 1$,
\ben\label{eq:FE}
\E(F_n(\theta) \mid {\cal F}_{n-1})= \frac{Z_n(e^{-im\theta })}{Z_{n-1}(e^{-im\theta })}\frac{nF_{n-1}(\theta) \Big(1+\frac{\Phi(\theta)}{n}\Big)}{n+1},
\een
where $\Phi(\theta)=\E(e^{i\theta \Delta})$ is the Fourier transform of $\Delta$.
We have assumed that $\Delta$ has exponential moments, and more precisely that there exists $r_1>0$ such that $S_{r_1} = \sup_{\theta \in [-r_1,r_1]} \l|\E(\exp(\theta \Delta))-1\r|<+\infty$. Let
\ben
D_{r_1}:=\{w \in \C, |\Im(w)|\leq r_1\}
\een
be the horizontal band centred around the $x$-axis, of width~$2r_1$. 
We have
\ben \label{eq:bound}
\sup_{z \in D_{r_1}} |\E(\exp(iz \Delta))|& \leq & \sup_{z \in D_{r_1}}|\E(\exp(i\Re(z)\Delta-\Im(z) \Delta))|\\
                   & \leq & \sup_{z \in D_{r_1}} \E(\exp(-\Im(z) \Delta)) \leq 1+S_{r_1}.
\een
From here, we infer that $\Phi$ is holomorphic on $D_{r_1/2}$.

From Equation~\eref{eq:FE} we get that, for all $n\geq 1$,
\ben\label{eq:Fn}
\E\l(F_{n}(\theta)\r)= \frac{Z_n(e^{-im\theta})}{n+1}\dis\prod_{j=1}^n \l(1+\frac{\Phi(\theta)}{j}\r)=\frac{1}{n+1}Z_n(e^{-i m\theta})\, Z_n(\Phi(\theta)+1),
\een
implying that $\theta\mapsto \E(F_n(\theta))$ is holomorphic on $D_{r_1/2}$. By the first statement of \eref{eq:ine2}, we deduce
\begin{lem} There exists $r_2\in(0,r_1/2)$, such that for any $z\in D_{r_2}$, for any $n \geq 1$,  $\mathbb E(F_n(z))$ is non-null. Hence, for  any $z\in D_{r_2}$,  $(T_n(z))_{n\geq 0}$ is a martingale. 
\end{lem}

\vspace{\baselineskip}
{\bf The BST height profile martingale -- }
In~\cite{CDJH}, the authors study a martingale~$\big({W_n(z)}/{\mathbb E W_n(z)}\big)_{n\geq 0}$ defined as follows: for all $z\in\mathbb C$, $W_n(z):= \sum_{k\geq 0} U_n(k) z^k$ where $U_n(k)$ is the number of leaves at height $k$ in the $n+1$-leaf random binary search tree.
This martingale is different from ours, but we have (by~\cite[Lemma~2]{CDJH}) 
\ben
\E(W_n(z))=Z_n(2z).
\een
To prove Theorem~\ref{theo:BRW}, we use Joffe, Le Cam \& Neveu~\cite{JLCN}'s method, 
many specific details being similar to those developed by Chauvin \& al.~\cite{CDJH}.
First of all, by \cite[Lemma~3]{CDJH}, 
$\dis \l|\E\l(W_n\l(z\r)\r)-\frac{n^{2z-1}}{\Gamma(2z)}\r|={\mathcal O}(n^{2\Re(z)-2})$ uniformly on all compact sets of $\C$, when $n\to+\infty$, so that uniformly on all compact sets of $\C$,
\ben\label{eq:Zn}
\l|Z_n\l(x\r)-\frac{n^{x-1}}{\Gamma(x)}\r|={\mathcal O}(n^{\Re(x)-2}).
\een
Thus, in view of Equation~\eqref{eq:Fn}, we have
\begin{lem}\label{lem:esperance_F_n}
Asymptotically when $n$ goes to infinity, uniformly for $z\in D_{r_2}$,
\ben\label{eq:zt2}
\E[F_{n}(z)]=\frac{n^{e^{-im z}+\Phi(z)-2}}{\Gamma(e^{-imz})\Gamma(\Phi(z)+1)}\,(1+o(1)).
\een
\end{lem}
\begin{proof}Since $\Delta$ has exponential moments, by \eref{eq:bound}, for 
$\{w=(1+\Phi(z))/2, z\in D_{r_2}\}$ is bounded, and using that $e^{-imz H_{n+1}}\sim n^{-imz}$. 
\end{proof}
We state the strong convergence of the renormalised random Fourier transform~$T_n$:
\begin{prop}\label{prop:cvps_T_n}
For any $\theta\in\mathbb R$,
\[T_n\l(\frac{\theta}{\sqrt{\log n}}\r) \as 1.\]
\end{prop}
The proof of this proposition is postponed: we first show how to prove Theorem~\ref{theo:BRW} from there.
\begin{proof}[Proof of Theorem~\ref{theo:BRW}]
Note that, for all $z\in\mathbb C$, letting $m_1=m$ and $m_2=\sigma^2+m^2$ the two first moments of $\Delta$,
\ben
\Phi(z)-2+e^{-imz} &=& 1+im_1z-\frac{m_2 z^2}2 -2 + 1 - i m_1 z - m_1^2\frac{z^2}2 + o(|z|^2)\\
&=& - (\sigma^2 + 2m^2)\frac{z^2}{2} + o(|z|^2),
\een 
when $|z|$ tends to zero. Thus, in view of Lemma~\ref{lem:esperance_F_n}, for all $\theta\in\mathbb R$, we have
\ben
\E \l[F_n\Big(\frac\theta{\sqrt{\log n}}\Big)\r]
=\frac{n^{-\frac{\theta^2(\sigma^2+2m^2)}{2\log n}}}{\Gamma(2)}(1+o(1)) 
\to  \exp\l(-\frac{\theta^2}2(\sigma^2+2m^2)\r).
\een
Thanks to Proposition~\ref{prop:cvps_T_n}, for all $\theta\in\mathbb R$, 
almost surely when $n$ tends to infinity, we have
\[T_n\Big(\frac{\theta}{\sqrt{\log n}}\Big) 
= (1+o(1))\, e^{\frac{\theta^2(\sigma^2+2m^2)}{2}} F_n\Big(\frac{\theta}{\sqrt{\log n}}\Big) 
\to 1,\]
which implies that
\ben\label{eq:cvfn}
F_n\Big(\frac{\theta}{\sqrt{\log n}}\Big) \as  \exp\l(-\frac{\theta^2}2(\sigma^2+2m^2)\r).\een
Note that the deterministic map $\theta\mapsto Z_n(e^{-im\theta})$ is the Fourier transform of the random variable 
\[K_n := - m \sum_{j=1}^n B_j\] 
where the $B_j$'s are independent Bernoulli random variables of respective parameters $1/j$. 
Since $K_n$ has mean $- m H_n^{\sss (1)} \sim -m\log n$ (where 
$H_n^{\sss (p)} = \sum_{k=1}^n k^{-p}$) and variance 
$\Var(K_n)=\sum_{j=1}^n 1/j(1-1/j)= H_n^{\sss (1)}-H_n^{\sss (2)}\sim \log n$, 
by Linderberg's theorem,
\ben 
\frac{K_n + m \log n}{\sqrt{\log n}} \dd \mathcal N(0,m^2),
\een 
which, by Lévy's continuity theorem is equivalent to  
\ben\label{eq:cvzn}
e^{im\theta\sqrt{\log n}} Z_n\l(e^{-{im\theta}/{\sqrt{\log n}}}\r)\to \exp\l(-{m^2 \theta^2}/2\r).
\een
Hence, since 
\ben\label{eq:er}
F_n\Big(\frac{\theta}{\sqrt{\log n}}\Big) &=& 
\l[e^{im\theta\sqrt{\log n}} Z_n\l(e^{-{im\theta}/{\sqrt{\log n}}}\r)\r] 
\l[e^{-im\theta\sqrt{\log n}} \sum_{k=1}^{n+1} \frac{\exp\l(iX_k\,\frac{\theta}{\sqrt{\log n}}\r)}{n+1}\r],
\een
using Equations~\eref{eq:cvfn},~\eref{eq:cvzn} and~\eref{eq:er}, 
we see that the Fourier transform of $\Theta_{\sqrt{\log n},\, m \log n}\,\l((n+1)^{-1}\mathcal M_{n+1}\r)$ 
given by the second bracket in the right-hand side of Equation~\eref{eq:er}
converges pointwise a.s.\ to the Fourier transform of $\mathcal N(0, \sigma^2+m^2)$. 
By Berti \& al. \cite[Theorem~2.6]{BPR}, this implies that 
\[\Theta_{\sqrt{\log n},\,m\log n}\,\l(n^{-1}\mathcal M_{n}\r)\as  \mathcal N(0, \sigma^2+m^2),\] 
which concludes the proof. 
\end{proof}

The end of the section is now devoted to proving Proposition~\ref{prop:cvps_T_n}. To do so, we follow the strategy used in~\cite{CDJH} and start by proving an equivalent of their Lemma~4.
For all $z, z_1, z_2\in\mathbb C$, set
\ben\label{eq:form1a}
f_n(z)&:=&(n+1) F_n(z)
\een
and 
\ben\label{eq:form1b}
\ovs{F_n}(z)&:=& \frac{F_n(z)}{Z_n(e^{-imz})},\\
\label{eq:form2}
\ovs{f_n}(z)&:=&(n+1) \ovs{F_n}(z),\\
\label{eq:form3}\ovs{P_n}(z_1,z_2)&:=&{\E[\ovs{f_n}(z_1)\ovs{f_n}(z_2)]}.
\een

\begin{lem}
For all $z_1, z_2\in\mathbb C$,
\[\ovs{P_{n+1}}(z_1,z_2)=\sum_{j=0}^n \l(\beta^{\star}_j(z_1,z_2)\prod_{k=j+1}^n\alpha^\star_k(z_1,z_2)\r)+\prod_{j=0}^n \alpha^\star_j(z_1,z_2), 
\]
 where 
\[
\alpha^\star_j(z_1,z_2)=1+\frac{\Phi(z_2)+\Phi(z_1)}{j+1},\]
and, for all $j\geq 0$,
\[\beta^{\star}_j(z_1,z_2)
=\frac{\E[\ovs{f_j}(z_1+z_2)]}{j+1}\,\Phi(z_1+z_2)
=\E[\ovs{F_j}(z_1+z_2)]\Phi(z_1+z_2).\]
\end{lem}
\begin{proof} 
To get the $(n+2)$-node RRT from the $(n+1)$-node RRT,
one chooses uniformly at random a node $U(n)$ in the $(n+1)$-node RRT 
and attaches a new child to this node.
Moreover, the branching random walk at this 
new node is the value of the walk at $U(n)$ plus an increment~$\Delta$. 
Thus, for all $n\geq 0$, 
\ben\label{eq:fr}
\ovs{f_{n+1}}(z)=\ovs{f_n}(z)+ e^{iz(X_{U(n)}+\Delta_{n+1})},
\een
where $(\Delta_i)_{i\geq 1}$ is a sequence of i.i.d.\ copies of $\Delta$.
We thus have
\ben
\ovs{P_{n+1}}(z_1,z_2)
=\E\bigg[\E\Big[\Big(\ovs{f_n}(z_1)+ e^{iz_1(X_{U(n)}+\Delta_{n+1})}\Big) 
\Big(\ovs{f_n}(z_2) + e^{iz_2(X_{U(n)}+\Delta_{n+1})}\Big)~\Big|~{\cal F}_n\Big]\bigg]
\een
Recall that $\mathcal M_{n+1}/(n+1)=(\sum_{k=1}^{n+1} \delta_{X_k})/(n+1)$ is the empirical distribution of the labels of the tree and thus, 
$\int e^{izx}\, \frac{d{\cal M}_{n+1}(x)}{n+1}= \ovs{f_n}(z)/(n+1)$.
We have
\ben
\ovs{P_{n+1}}(z_1,z_2)
&=&\E\bigg[\int\Big(\ovs{f_n}(z_1)+e^{iz_1(x+\Delta)}\Big)
\Big(\ovs{f_n}(z_2)+e^{iz_2(x+\Delta)}\Big) \frac{d{\cal M}_{n+1}(x)}{n+1}\bigg]\\
&=& \E\bigg[\ovs{f_n}(z_1)\ovs{f_n}(z_2) + \ovs{f_n}(z_1)\ovs{f_n}(z_2)\frac{\Phi(z_2)+\Phi(z_1)}{n+1}
 + \ovs{f_n}(z_1+z_2)\frac{\Phi(z_1+z_2)}{n+1}\bigg],
\een 
which implies
\ben\label{eq:Pn}
\ovs{P_{n+1}}(z_1,z_2)&=&\ovs{P_n}(z_1,z_2)\l(1+\frac{\Phi(z_2)+\Phi(z_1)}{n+1}\r)
+\E[\ovs{F_n}(z_1+z_2)]{\Phi(z_1+z_2)},
\een
so that
\ben\label{eq:qfz}
\ovs{P_{n+1}}(z_1,z_2)&=&\ovs{P_n}(z_1,z_2)\alpha^\star_n(z_1,z_2)+\beta^{\star}_n(z_1,z_2).
\een
A simple recursion concludes the proof.
\end{proof} 
Up till now, we have restricted our study to $z\in D_{r_2}$ 
(the band centred around the vertical axis, and of width $2r_2$) 
on which $T_n$ is well defined for each $n\geq 1$. 
By \eref{eq:ine2}, there exists $r_3\in (0,+\infty)$ such that $S_{r_3}<\nicefrac12$. 
Let
\[r_4= \min \{r_2,r_3\}.\]
\begin{prop}\label{pro:cvm}
 There exists a closed ball~$B$ centred at 0 in $\C$, 
such that for any~$z$ in~$B$, the martingale $(T_n(z),n\geq 0)$ converges in $L^2$. 
The convergence of $(T_n,n\geq 0)$ holds almost surely in $C(B,\C)$ 
(the set of continuous functions on $B$ taking their values in $\C$, 
equipped with the topology of uniform convergence).  
\end{prop}
In fact we prove that the random function~$T_n$ converges uniformly to a random holomorphic function~$T$ on $B$.
\begin{proof}
For all $z_1,z_2\in D_{r_4}$, we have, when~$j$ and~$n$ both go to infinity, that
\begin{align*}
\prod_{k=j+1}^n \alpha_k^{\star}(z_1,z_2)
&= \prod_{k=j+1}^n \left(1+ \frac{\Phi(z_1)+\Phi(z_2)}{k+1}\right)
= \exp\bigg([\Phi(z_1)+\Phi(z_2)]\bigg(\sum_{k=j+1}^n \frac1{k+1}+ \mathcal O(\nicefrac1{j})\bigg)\bigg)\\
&= \exp\bigg([\Phi(z_1)+\Phi(z_2)]\Big(\log\nicefrac nj  + \mathcal O(\nicefrac1j)\Big)\bigg),
\end{align*}
by Euler's formula for Harmonic sums.
Using the fact that $z_1, z_2\in D_{r_4}$, we thus get
\[\prod_{k=j+1}^n \alpha_k^{\star}(z_1,z_2) = \Big(\frac n j\Big)^{\Phi(z_1)+\Phi(z_2)} \big(1+\mathcal O(\nicefrac1j)\big).\]
Moreover, using Lemma~\ref{lem:esperance_F_n}, we have
\ben
\beta^\star_n(z_1,z_2)
&=&\frac{\Phi(z_1+z_2)}{\Gamma(1+\Phi(z_1+z_2))}n^{\Phi(z_1+z_2)-1}
+\mathcal O\big(n^{\Re(\Phi(z_1+z_2)-2}\big).
\een
We have
\begin{align*}
|\ovs{P_{n}}(z_1,z_2)|
&=\l|\sum_{j=0}^{n-1} \bigg(\beta^{\star}_j(z_1,z_2)\prod_{k=j+1}^{n-1}\alpha^\star_k(z_1,z_2)\bigg)+\prod_{j=0}^{n-1} \alpha^\star_j(z_1,z_2)\r|\\
&\leq \l|\frac{\Phi(z_1+z_2)}{\Gamma(1+\Phi(z_1+z_2))}\r|
\,\sum_{j=0}^{n-1}\l(j^{\Re(\Phi(z_1+z_2))-1}+\mathcal O\big(j^{\Re(\Phi(z_1+z_2))-2}\Big)\r)
\Big(\frac nj\Big)^{\Re(\Phi(z_1)+\Phi(z_2))}\big(1+\mathcal O(\nicefrac1j)\big)\\
& +n^{\Re(\Phi(z_1)+\Phi(z_2))}\big(1+\mathcal O(\nicefrac1n)\big)\\
&=\l|\frac{\Phi(z_1+z_2)}{\Gamma(1+\Phi(z_1+z_2))}\r|\, n^{\Re(\Phi(z_1)+\Phi(z_2))}\bigg(1+\sum_{j=0}^{n-1} j^{\Re(\Phi(z_1+z_2) - 1  -\Phi(z_1)-\Phi(z_2))}\bigg)(1+o(1)),
\end{align*}
when $n$ tends to infinity. Thus in view of Lemma~\ref{lem:esperance_F_n} and Equations~\eref{eq:form1a}, \eref{eq:form1b}, \eref{eq:form2} and  \eref{eq:form3} we have that
\begin{align}
&\l|\frac{\E\big[F_n(z_1)F_n(z_2)\big]}{\E[F_n(z_1)]\E[F_n(z_2)]}\r|
= \frac{\ovs{P_n}(z_1,z_2)}{ \E[\ovs{F_n}(z_1)]\E[\ovs{F_n}(z_2)]}\notag\\
&\hspace{1cm}=
\bigg|\frac{\Gamma(1+\Phi(z_1))\Gamma(1+\Phi(z_2))\Phi(z_1+z_2)}{\Gamma(1+\Phi(z_1+z_2))} \bigg|
\bigg(1+\sum_{j=0}^{n-1} j^{\Re(\Phi(z_1+z_2)-1  -\Phi(z_1)-\Phi(z_2))}\bigg)\,(1+o(1)),\label{eq:zrgte}
\end{align}
Note that the first term in the above product is uniformly bounded for $z_1$ and $z_2$ in $D_{r_4}$. Since, for any $z$,
\[\E\l(|T_n(z)|^2\r)=\l|\frac{\E\big[F_n(z)\ov{F_n(z)}\big]}{\E[F_n(z)]\E[\ov{F_n(z)}]}\r|.\]
For all $a,b\in \mathbb R$, $\ov{F_n(a+ib)}=F_n(-a+ib)$, implying that, by Equation~\eref{eq:zrgte}, 
the martingale $T_n(a+ib)$ is uniformly bounded in $L^2$ if
\ben\label{eq:ezg}
\Re(\Phi(2ib) -\Phi(a+ib)-\Phi(a-ib)) <0.\een
Note that this last condition holds for all $z=a+ib$ in a rectangle $R$ containing $0$ in its interior (and included in $D_{r_4}$), since $\Phi(0)=1$ and since $\Phi$ is continuous at~0. Hence, for all $z\in R$, the martingale $T_n(z)$ converges a.s.; this is a consequence of the $L^2$-boundness, which implies $L^2$ convergence (see e.g.\ \cite[Theorem~4]{chatterji1960}). Finally, recall that in any Banach space, a martingale which converges in $L^2$ also converges a.s.\ (see e.g.\ Pisier~\cite[Theorem~1.14]{Pisier}); therefore, for all $z\in R$, $T_n(z)$ converges a.s.\ and we denote by~$T(z)$ its limit.

Let us now discuss the convergence of the process $T_n \as T$ on $R$, in a convenient functional space. The above discussion concerning the convergence at any fixed $z\in R$ 
implies straightforwardly the a.s.\ joint convergence of $(T_n(z_j),1\leq  j \leq k)$ to $(T(z_j),1\leq  j \leq k)$, for all integers $k\geq 1$ and $(z_1, \ldots, z_k)\in R^k$. 

The a.s.\ convergence of $(X_n,Y_n)$ to $(X,Y)$ implies the convergence of $X_nY_n$ to $XY$, and here, since $T_n(z_1)$ and $T_n(z_2)$ converge in $L^2$, we get that $T_n(z_1)T_n(z_2)\to T(z_1)T(z_2)$ a.s.\ and in~$L^1$, so that
\ben
\Gamma_n(z_1,z_2):=\E[T_n(z_1)T_n(z_2)]\to \Gamma(z_1,z_2):=\E[T(z_1)T(z_2)].
\een
From \eref{eq:zrgte}, we see that $\Gamma_n$ converges normally for $(z_1,z_2)\in R^2$, and since $(z_1,z_2)\mapsto \Gamma_n(z_1,z_2)$ is holomorphic, we deduce that its limit  $\Gamma$ is holomorphic too.

We face then a situation where the sequence of continuous processes $(T_n)_{n\geq 0}$ converges uniformly to~$T$ on~$R$. Now, since $\Gamma$ is holomorphic, we have
\ben
\E\l[\l|T(x )-T(y)\r|^2\r] \leq |{\Gamma(x,-x)}+{\Gamma(y,-y)}-\Re\l({\Gamma(x,-y)}+{\Gamma(-x,y)}\r)|\leq c |x-y|^2
\een
for some $c>0$, uniformly on $R$. 
By Kolmogorov criterion, the function $T$ admits a continuous modification on $R$. 
Finally, since $T_n$ is continuous for all $n$, we get that $T_n\to T$ in $C(R,\mathbb{C})$.
\end{proof}

\begin{proof}[Proof of Proposition \ref{prop:cvps_T_n}.]
From Proposition \ref{pro:cvm}, we infer that $T_n\dd T$ in  $C(R\cap \R,\C)$, that is $(T_n(\theta), \theta \in R\cap \R)\dd (T(\theta),\theta \in R\cap \R)$ for the topology of uniform convergence. 
Moreover, since $T_n(0)=1$ for all integers~$n$, we have $T(0)=1$.
Finally, since $T$ is continuous on $R\cap \R$, we get that, for all $\theta\in \R$,
\[T_n\Big(\frac{\theta}{\sqrt{\log n}}\Big) \as T(0)=1.\]
\end{proof}

\subsection{Higher dimension}
To prove Theorem~\ref{theo:BRW} in dimension $d\geq 2$, one can 
\begin{itemize}
\item either adapt the one-dimensional proof to dimension~$d$. 
This is done by considering $d$-dimensional Fourier transforms instead: take 
\[F_n(\theta)= Z_n(e^{-i\,m\cdot\theta})\sum_{k=1}^{n+1} \frac{e^{i\, \theta \cdot X_k}}{n+1},\]
for all $\theta\in\mathbb R^d$,
where $m\cdot\theta$ stands for the scalar product of $m$ and $\theta$.
The definition of $Z_n(\cdot)$ remains unchanged.
The main change to make in the above proof is in the proof of Theorem~\ref{theo:BRW} itself
where one needs to note that $Z_n(e^{i\, m\cdot\theta})$ is the Fourier transform of
\[K_n = (K_{n1}, \ldots, K_{nd}), \quad\text{ where } K_{ni}= - m_i \sum_{k=1}^n B_k,\]
and $B_k$ is a Bernoulli-distributed random variable of parameter $1/k$ (and the $B_k$ are independent).  
Note that $K_n$ has mean $-H_n^{(1)}m$ where $H_n^{(p)} = \sum_{k=1}^n k^{-p}$ 
and variance $\Var(K_n)= (H_n^{(1)}-H_n^{(2)}) m^T m$. 
Then, by Linderberg theorem, using that $H_n^{(1)} \sim \log n$ when $n$ tends to infinity, we have
\ben 
\frac{K_n + m \log n}{\sqrt{\log n}} \dd \mathcal N(0,m^T m).
\een 
Which, by Lévy's continuity theorem is equivalent to  
\ben
e^{i\sqrt{\log n}(m\cdot\theta)} Z_n\l(e^{-i\,m\cdot\theta/\sqrt{\log n}}\r)\to \exp\l(-\frac{\theta^T (m^T m) \theta}{2}\r).
\een
This replaces Equation~\eref{eq:cvzn}.
The rest of the proof can be adapted straightforwardly.
\item or make the following remark: note that the Fourier transform of $n^{-1}\mathcal M_n$ at $\theta\in\mathbb R^d$ verifies
\[F_n(\theta)=\frac1n \sum_{k=1}^n e^{i\|\theta\| \,\frac{\theta}{\|\theta\|}\cdot X_k},\]
which is the Fourier transform of $n^{-1} \sum_{k=1}^n \delta_{u \cdot X_k}$, where $u:= \theta/\|\theta\|$, taken at $\|\theta\|$.
We can thus apply the one-dimensional result to the MVPP associated to the random walk of increment $u\cdot\Delta$. Note that
\[\mathbb E[u\cdot\Delta] = m\cdot u,
\quad \text{ and }\quad
\Var(u\cdot\Delta) = u^T \Sigma^2 u - (u\cdot m)^2.\]
Thus
\[F_n\Big(\frac{\theta}{\sqrt{\log n}}\Big) 
\to e^{-\big(u^T \Sigma^2 u + (m\cdot u)^2\big){\|\theta\|^2}/{2}}
= e^{-{\big(\theta^T \Sigma^2 \theta + (m\cdot \theta)^2\big)}/2},\]
when $n$ tends to infinity, which proves the $d$-dimensional statement.
\end{itemize}

\section{Proof of Theorem \ref{th:main-2} (without-replacement case)}\label{sec:without_replacement}

In the without-replacement case, when a ball of colour~$x$ is drawn, 
it is removed from the urn, and replaced by $\kappa$ balls, whose colours are represented by ``the $\kappa$ atoms'' of the measure ${\cal R}_x$.
Following what is done in the previous sections, 
we encode the urn process by a sequence of BMCs associated to a sequence of growing trees.
A similar idea has already been used in the literature to encode $d$-colour P\'olya urns as a tool to obtain fixed point equations 
(see Knape \& Neininger~\cite{KN} and Chauvin \& al.~\cite{CMP}).

The idea is the following:
At time~0, the tree is reduced to the root $\varnothing$ labelled $L_\varnothing=c$, 
the colour of the unique ball in the urn at time 0.
At time~$n$, i.e.\ after~$n$ drawings, there are $1+n(\kappa-1)$ balls in the urn. 
The urn at time~$n$ is represented by a tree with~$n$ internal nodes, 
where each internal node has~$\kappa$ children. 
The labels of the leaves correspond to the colour of the balls in the urn at time~$n$, 
and the labels of the internal nodes, corresponds to the colour of balls that have 
been in the urn in the past, and which have been drawn and removed from the urn before time~$n$.
Choosing a ball~$b$ uniformly corresponds to choosing a leaf~$u$ uniformly at random in the tree. 
The withdrawal of the chosen ball $b$ and the addition of $\kappa$ new balls $b_1, \ldots, b_{\kappa}$ 
is encoded by adding $\kappa$ nodes to the tree, being the children $(u0,\ldots,u(\kappa-1))$ of~$u$.
As done in the with-replacement case, we now formalise this idea by coupling the MVPP with a BMC.

{\bf The random recursive $\kappa$-ary tree -- }
This random tree is defined as a Markov chain $\mathtt T_n^{\kappa}$ on the set of rooted trees 
whose nodes all have either 0 or $\kappa$ children (also called $\kappa$-ary trees). 
The tree~$\mathtt T_0^{\kappa}$ is by definition equal to $\{\varnothing\}$. 
Given $\mathtt T_n^{\kappa}$, we build $\mathtt T_{n+1}^{\kappa}$ as follows: 
take a node $u_n$ at random among the set of leaves of $\mathtt T_n^{\kappa}$, and let $\mathtt T_{n+1}^{\kappa} = \mathtt T_n^{\kappa} \cup \{u_n0, \ldots, u_n(\kappa-1)\}$.

Note that (taking $\kappa=2$) $\CBST_n = \mathtt T_n^{2}$ in distribution.

\bigskip
{\bf The enriched model -- }
As for the binary search tree, it is useful to build the enriched random recursive $\kappa$-ary tree as follows.
Recall that the Dirichlet distribution of parameters $\kappa$ and $\alpha$ has density
\[d\mu_{\kappa,\alpha}(x_1,\ldots,x_\kappa)= \frac{\Gamma(\kappa\alpha)}{\Gamma(\alpha)^\kappa}\,\prod_{i=1}^\kappa x_i^{\alpha-1}\]
on the simplex $\Delta_{\kappa-1}=\{(x_1, \ldots, x_\kappa)\colon x_i \geq 0, \sum x_i=1\}$.
In the following, we will take $\alpha=\beta-1=1/(\kappa-1)$.

With each node $u$ of the complete $\kappa$-ary tree, 
associate a random variable $X_u=(X_{u}^{\sss (1)},\ldots,X_{u}^{\sss (\kappa)})\sim \mu_{\kappa,\alpha}$. 
Using these variables, we associate to each node an interval: 
the interval associated to the root is $I_\varnothing=[0,1]$. 
To its children $1\leq j \leq \kappa$, it is
\[I_j=\l[X_{\varnothing}^{\sss (1)} + \cdots+ X_{\varnothing}^{\sss (j-1)},X_{\varnothing}^{1}+\cdots+X_\varnothing^{\sss (j)}\r],\]
with $X_\varnothing^{\sss (\kappa)} = 1$, 
so that, $(I_1, \ldots, I_\kappa)$ forms a partition of $[0,1]$, and $|I_j|=X_\varnothing^{\sss (j)}$. 
We proceed similarly, recursively: the intervals $(I_{u1}, \ldots, I_{u\kappa})$ associated to 
the children of~$u$ are obtained by forming a partition of $I_u$ in $\kappa$ parts, 
the variables $X_u^{\sss (j)}$ giving the proportion of the $j$th part: formally, 
if $I_u=[a,b]$, then 
\[I_{uj}=\l[a+(b-a)(X_{u}^{\sss (1)}+\cdots+X_{u}^{\sss (j-1)}), a+(b-a)(X_{u}^{\sss (1)}+\cdots+X_{u}^{\sss (j)})\r],\]
with $X_u^{\sss (\kappa)} = 1$.
Hence, following the sequence of intervals along a branch starting at the root, one sees a sequence of nested intervals. 

We build the tree $\mathtt G_n$ as follows:
Let $(U_i)_{i\geq 1}$ be a sequence of i.i.d.\ random variable, uniform on $[0,1]$.
Let $\mathtt G_0 = \{\varnothing\}$.
Given $\mathtt G_n$, we define $\mathtt G_{n+1}$ as follows:
let 
\[\partial \mathtt G_n := \{uj \colon u\in\mathtt G_n \text{ and }0\leq j\leq \kappa-1 \text{ and }uj\notin \mathtt G_n\}.\]
Let $u_n$ be the node of $\mathtt G_n$ such that $U_n\in I_{u_n}$. We set $\mathtt G_{n+1} = \mathtt G_n \cup \{u_n\}$.

\begin{lem}\label{lem:enriched}
We have in distribution $(\mathtt G_n)_{n\geq 0} = (\mathtt T_n^{\kappa})_{n\geq 0}$.
\end{lem}

\begin{proof}
The proof that this representation is exact can be found in \cite[Prop 20]{alb-mar} for example. 
We detail it here for completeness' sake. It is enough to prove that, for all $n\geq 1$, 
the sizes of the $\kappa$ subtrees of the root of $\mathtt G_n$ have the same distribution as the sizes of the subtrees of the $\kappa$ subtrees of the root of $\mathtt T_n^{\kappa}$.

Note that the size of the $j$th subtree of the root of $\mathtt G_n$ is given by
\[N_j^{\sss (n)}=\{2\leq m \leq n \colon U_m \in I_j\}.\]
We let$(\overline{N}^{\sss (n)}_j, 1\leq j \leq \kappa)$ be the size of the $j$th subtree of the root in $\mathtt T_n^{\kappa}$. 
Our aim is to prove that, for all integers $n\geq 1$,
\[\l(N_1^{\sss (n)}, \ldots, N_\kappa^{\sss (n)}\r)\eqd \l(\overline{N}^{\sss (n)}_1, \ldots, \overline{N}^{\sss (n)}_\kappa\r).\]
For all integers $n_1, \ldots, n_\kappa$ such that $\sum n_j =n-1$, we have
\begin{align}
\P\l(\l(N_1^{\sss (n)}, \ldots, N_\kappa^{\sss (n)}\r)=(n_1, \ldots, n_\kappa)\r)
&=
\int_{\Delta_{d-1}}\binom{n-1}{n_1,\ldots,n_\kappa} 
\l(\prod_{i=1}^\kappa x_i^{n_i}\r)  d\mu_{\kappa,\beta-1}(x_1,\ldots,x_\kappa)\notag \\
&=\binom{n-1}{n_1, \ldots, n_\kappa} 
\frac{\prod_{i=1}^\kappa \Gamma(n_i+\beta-1)}{\Gamma(n-1+\kappa(\beta-1))}\,\frac{\Gamma(\kappa(\beta-1))}{\Gamma(\beta-1)^{\kappa}}\label{eq:ethh}
\end{align}
Note that for any set of $(t_1, \ldots, t_n)$ of $\kappa$-ary trees verifying $|t_i| = i$ and $t_1\subset t_2\subset\cdots \subset t_n$, we have
\[\mathbb P(\mathtt T_1^{\kappa} = t_1, \ldots, \mathtt T_n^{\kappa} = t_n) 
= \prod_{i=1}^{n-1} \frac{1}{1+i(\kappa-1)}
=\frac{1}{\Gamma(2+(n-1)(\kappa-1))}.\]
The number of sets $(t_1, \ldots, t_n)$ of $\kappa$-ary trees verifying $|t_i| = i$ and $t_1\subset t_2\subset\cdots \subset t_n$ and such that the $\kappa$ subtrees of the root of $t_n$ have respective sizes $n_1, \ldots, n_\kappa$ is given by
\[\binom{n-1}{n_1, \ldots, n_\kappa} \prod_{i=1}^\kappa K(n_i),\]
where $K(m)$ is the number of different $\kappa$-ary trees of size $m$, for all integer $m$ (we may also describe directly the subtrees size distribution).
Given that (see for example~\cite[page~68]{FS})
\ben\label{eq:sdg}
K(m)
=\frac1{(\kappa-1) m +1} \binom{\kappa m}{m}
=\frac{\Gamma(\kappa m+1)}{\Gamma(m+1)\Gamma((\kappa-1)m+1)},\een
we get 
\ben\label{eq:jht}
\P\l(\l(\overline{N}_1^{\sss (n)}, \ldots, \overline{N}_\kappa^{\sss (n)}\r)=(n_1, \ldots, n_\kappa)\r)
=\binom{n-1}{n_1,\ldots, n_\kappa} \frac1{\Gamma(1+n(\kappa_1))} \prod_{i=1}^\kappa \frac{\Gamma(\kappa n_i+1)}{\Gamma(n_i+1)\Gamma((\kappa-1)n_i+2)}
\een 
Expanding the terms $\Gamma(n_i+\beta-1)$ in \eref{eq:ethh}, using that for all integers $m$,
\ben\label{eq:ehet}
\Gamma(m+\beta-1)
=\Gamma\l(m+\frac1{\kappa-1}\r)
=\Gamma\l(\frac{1}{\kappa-1}\r)\prod_{j=0}^{m-1} \frac{(\kappa-1)j+1}{\kappa-1}.
\een
We then see that the quantities $\P\Big(\big(N_1^{(n)},\ldots,N_\kappa^{(n)}\big) = (n_1,\ldots,n_\kappa)\Big)$ and $\P\Big(\big(\overline{N}_1^{(n)},\ldots,\overline{N}_\kappa^{(n)}\big) = (n_1,\ldots,n_\kappa)\Big)$ are proportional, 
and thus equal.
\end{proof}

\bigskip
{\bf The associated BMC -- }
The BMC associated with our $\kappa$-discrete MVPP relies on the fact that we have assumed that $\kappa \mathcal R_x$ is the sum of $\kappa$ Dirac masses (see Equation~\eref{eq:nb}). In other words, for all $x\in\PP$, the replacement measure $\mathcal R_x$ can be re-written as
\[\mathcal R_x = \frac1\kappa\, \sum_{j=1}^{\kappa} \delta_{y_j(x)},\]
where $y_1(x), \ldots, y_\kappa(x)$ are $\kappa$ (not necessarily distinct) elements of~$\PP$.
The idea behind the form of the BMC would that if a node is labelled by~$x$, 
its children should be labelled by $y_1(x), \ldots, y_{\kappa}(x)$.
But in order for the label along a branch to be a Markov chain which does not depend on the rank of the ancestors in their siblings but only on their depth, 
one should randomly shuffle the labels of siblings: for all $y=(y_1, \ldots, y_\kappa)\in\mathbb P^\kappa$, we let
\ben
{\sf Sym}(y)=\frac{1}{\kappa!}\,\sum_{\sigma\in{\cal S}(\kappa)} \delta_{(y_{\sigma(1)}, \ldots ,y_{\sigma(\kappa)})}
\een
the probability measure which is the uniform distribution on all orderings of the multiset $y$ (${\cal S}(\kappa)$ denotes the symmetric group on $\{1,\ldots,\kappa\}$). For all $x\in\PP$, we denote by 
\ben
{\sf RM}_x:= {\sf Sym}(y_1(x), \ldots, y_{\kappa}(x)).
\een
\begin{lem}
Let $X(\mathtt T_n^{\kappa})$ be the BMC on the random $\kappa$-ary recursive tree
of initial distribution~$\mathcal M_0$ and kernel
\[K(x,~\cdot~) = {\sf RM}_x \quad(\forall x\in\PP).\]
Then the process defined for all integers~$n$ by
\[\mathcal M_n^{\star} = \frac1{1+n(\kappa-1)} \sum_{u\in L(\mathtt T_n^{\kappa})} \delta_{X(u)}\]
satisfies
$(\mathcal M_n^{\star})_{n\geq 0} = (\mathcal M_n)_{n\geq 0}$.
\end{lem}

\begin{rem} It is worth stressing on an important difference between the drawing without replacement case and the general case. In this latter case, the measure ${\cal M}_n$ is encoded by the node-values of the BMC  ${\cal M}_n={\cal M}_0+\sum_{u \in \RRT_n} \delta_{\RR_{X_u}}$. In the without-replacement case models (see again Remark~\ref{rem:cond}), 
the measure ${\cal M}_n$ is encoded by the leaves-values of the $\kappa$-ary tree.
\end{rem}
Following the same strategy as in the with-replacement case, we now state and prove the equivalent of Proposition~\ref{prop:inc-tree} for the random recursive $\kappa$-ary tree. Note that the random recursive $\kappa$-ary search tree has been studied in the literature for two particular values of~$\kappa$: as already mentioned, $\kappa=2$ corresponds to the random binary search tree, and the ternary case has been studied for example by Bergeron \& al.~\cite{Bergeron92varietiesof} and Albenque \& Marckert~\cite[Section 5.1]{alb-mar}. 
Following Example~1 (page 7) and Theorem~8 in Bergeron \& al.~\cite{Bergeron92varietiesof}, 
the height $H_n$ of a random node in $\mathtt T_n^{\kappa}$, follows a central limit theorem: for $\beta=1+1/(\kappa-1),$ we have 
\ben \label{eq:qgge}
\frac{H_n-\beta\log(n)}{\sqrt{\beta\log(n)}}\dd {\cal N}(0,1).
\een
\begin{prop} Let $U_n$ and $V_n$ be two uniform random nodes taken in $\mathtt T_n^{\kappa}$, we have
\ben
\l( \frac{|U_n|-\beta\log(n)}{\sqrt{\beta \log(n)}}, \frac{|V_n|-\beta\log(n)}{\sqrt{\beta \log(n)}},|U_n\wedge V_n|\r)\dd (\Lambda_1,\Lambda_2,K)
\een where  the three random variables are independent, 
$\Lambda_1$ and $\Lambda_2$ are two standard Gaussian random variables and $K$ is almost surely finite.
\end{prop}
\begin{proof} 
Using Lemma~\ref{lem:enriched}, one can adapt the arguments given in the proof of Proposition~\ref{prop:inc-tree-BST} (using again the enriched version by the Dirichlet random variables). We do not give the details.
\end{proof}

The rest of the proof is very similar to that of Theorem~\ref{th:main};
in particular, we couple the MVPP with a BMC on the $\kappa$-ary search tree 
using the following kernel:
\ben
\overline{K}(x,A_1\times\cdots\times A_\kappa)={\sf RM}_x(A_1 \times \cdots \times A_\kappa).
\een
Note that, under this kernel, the sequence of the labels given by the BMC to the nodes along a branch of the $\kappa$-ary tree (starting from the root) have the same distribution as a Markov chain of kernel $\mathcal R$. We do not give more details.
\small
\bibliographystyle{amsplain}
\bibliography{bib}
\end{document}